\documentclass[a4paper,12pt]{amsart}
\usepackage{amsaddr}
\usepackage{calrsfs}
\usepackage{graphicx}
\usepackage{float}
\usepackage{algorithm}
\usepackage{algpseudocode}
\usepackage[colorlinks]{hyperref}
\usepackage{array}
\numberwithin{equation}{section}

\theoremstyle{definition}
\newtheorem{dfn}{Definition}[section]
  
\theoremstyle{plain}
\newtheorem{thm}{Theorem}[section]
\newtheorem{pro}{Proposition}[section]

\theoremstyle{definition}
\newtheorem{rem}{Remark}[section]

\newcommand{\R}{\mathbb{R}}
\newcommand{\E}{\mathbb{E}}
\renewcommand{\P}{\mathbb{P}}

\newcommand{\e}{\mathrm{e}}

\renewcommand{\d}{\mathrm{d}}

\newcommand{\x}{\mathbf{x}}
\newcommand{\z}{\mathbf{z}}
\newcommand{\X}{\mathbf{X}}
\newcommand{\Z}{\mathbf{Z}}
\newcommand{\bxi}{\boldsymbol{\xi}}
\renewcommand{\i}{\mathrm{i}}

\begin{document}
\title[Walk-on-Cubes Simulation]
{Walk-on-Cubes Monte Carlo Simulation for nonisotropic fractional Laplace, Helmholtz, and Yukawa equations}

\date{\today}

\author[Rasila]{Antti Rasila}
	\address{Department of Mathematics with Computer Science, Guangdong Technion - Israel
		Institute of Technology, Shantou, Guangdong 515063, P. R. China \vskip0.025cm Department of Mathematics, Technion - Israel
		Institute of Technology, Haifa 3200003, Israel}

\email{antti.rasila@iki.fi; antti.rasila@gtiit.edu.cn}

\author[Sottinen]{Tommi Sottinen}
\address{School of Technology and Innovations, University of Vaasa, P.O. Box 700, FIN-65101 Vaasa, Finland}
\email{tommi.sottinen@uwasa.fi}

\author[Yuan]{Yaotong Yuan}
	\address{Department of Mathematics with Computer Science, Guangdong Technion - Israel
		Institute of Technology, Shantou, Guangdong 515063, P. R. China \vskip0.025cm Department of Mathematics, Technion - Israel
		Institute of Technology, Haifa 3200003, Israel}
\email{yuan09517@gtiit.edu.cn}

\begin{abstract}
We study the nonisotropic fractional analogs of Laplace, Helmholtz and Yukawa equations.  We provide a Duffin correspondence for the Yukawa equation and a Feynman--Kac reconstruction for the Helmholtz equation.  The foundation of our analysis is the fact that the nonisotropic fractional Laplace equation is related to a symmetric $\alpha$-stable L\'evy process with independent identically distributed components. By using this relation we provide a Walk-on-Cubes algorithm that simulates the solutions of Helmholtz and Yukawa equations.
\end{abstract}



\maketitle

%%%%%%%%%%%%%%%%%%%%%%%%%%%%%%%%%%%%%%%%%%%%%%%%%%%%%%%%%%%%%%%%%%%%%%%%%%%%%%%
%%%%%%%%%%%%%%%%%%%%%%%%%%%%%%%%%%%%%%%%%%%%%%%%%%%%%%%%%%%%%%%%%%%%%%%%%%%%%%%
\section{Introduction}

Let $\x =(x_1,\ldots,x_d)\in\R^d$. Let $\alpha\in(0,2)$ and let $A_\x^\alpha$ be the nonisotropic fractional Laplacian:
$$
A^\alpha_\x = \sum_{i=1}^d -(-\Delta_{x_i})^{\alpha/2},
$$
where $-(-\Delta_{x_i})^{\alpha/2}$ is the one-dimensional fractional Laplacian defined by the singular integral
\begin{equation}\label{eq:flaplace}
(-\Delta_{x_i})^{\alpha/2}f(x_i)
= 
C_\alpha\int_{-\infty}^\infty \frac{f(x_i)-f(y)}{|x_i-y|^{1+\alpha}}\, \d y.
\end{equation}
Here the positive one-dimensional normalization is
\begin{equation}\label{eq:calpha-definition}
C_\alpha
:=\frac{2^\alpha \Gamma((1+\alpha)/2)}{\pi^{1/2}|\Gamma(-\alpha/2)|}
=\frac{\alpha 2^{\alpha-1}\Gamma((1+\alpha)/2)}
{\pi^{1/2}\Gamma(1-\alpha/2)}.
\end{equation}

Let $D\subset \R^d$ be a bounded domain, and let $\lambda\in\R$. We consider the Dirichlet problem of finding for a given $g\colon D^c\to\R$ a solution $u\colon D\to \R$ such that 
\begin{eqnarray}\label{eq:fractional-dirichlet}
A^\alpha_\x u(\x) &=& \lambda u(\x) \quad\mbox{ if } \x\in D, \\  
\label{eq:fractional-dirichlet-boundary}
u(\x) &=& g(\x) \quad\mbox{ if } \x\in D^c.
\end{eqnarray}
If $\lambda<0$ \eqref{eq:fractional-dirichlet}--\eqref{eq:fractional-dirichlet-boundary} is called the fractional Helmholtz equation, if $\lambda>0$ it is called the fractional Yukawa equation, and for $\lambda=0$ it is called the fractional Laplace equation.

Recently Kyprianou et al. \cite{Kyprianou-Osojnik-Shardlow-2018} studied the isotropic fractional Laplace equation and provided the Walk-on-Spheres algorithm for its simulation.  Our work is related to theirs, with two major differences: first, we consider the nonisotropic case corresponding to independent identically distributed component L\'evy process as opposed to the isotropic L\'evy process.  Second, we also consider fractional Helmholtz and Yukawa equations.

To our knowledge, the nonisotropic operator $A^\alpha_\x$ is relatively unstudied in the literature.  One related study is Dybiec and Szczepaniec \cite{Dybiec-Szczepaniec-2015}.

The rest of the paper is organized as follows: In Section 2 we provide the Duffin correspondence relating the fractional Laplace equation and the fractional Yukawa equation. In Section 3 we provide the theoretical background for our Monte Carlo simulation, including the obstruction to the same spatial Duffin lifting in the Helmholtz regime. Finally, in Section 4 we introduce our Walk-on-Cubes algorithm and provide a simulation. 

%%%%%%%%%%%%%%%%%%%%%%%%%%%%%%%%%%%%%%%%%%%%%%%%%%%%%%%%%%%%%%%%%%%%%%%%%%%%%%%
%%%%%%%%%%%%%%%%%%%%%%%%%%%%%%%%%%%%%%%%%%%%%%%%%%%%%%%%%%%%%%%%%%%%%%%%%%%%%%%
\section{Duffin Correspondence}

Duffin \cite{Duffin-1971} provided a correspondence that transforms the classical Dirichlet problem of the Yukawa equation into a Dirichlet problem of the Laplace equation.  This was later studied and extended to the Helmholtz equation in \cite{Rasila-Sottinen-2018,Yang-Rasila-Sottinen-2017,Yang-Rasila-Sottinen-2019}.

Since the operator $A^\alpha_\x$ has a ``coordinate-sum'' form, we can provide the Duffin correspondence for it in a similar way as for the classical Laplacian.  The key ingredient for the Duffin correspondence is to find the non-zero solution $f=f_\lambda$ to the one-dimensional eigenvalue problem
\begin{equation}\label{eq:eigenfunction}
(-\Delta)^{\alpha/2}_yf(y) = \lambda f(y) \quad\mbox{for}\quad y\in\R.
\end{equation}
For the Yukawa regime ($\lambda > 0$), \eqref{eq:eigenfunction} has the explicit pointwise solution $f(y)=\cos(\lambda^{1/\alpha}y)$. For the Helmholtz regime ($\lambda < 0$), the heavy-tailed polynomial decay of the fractional operator rules out bounded analytical solutions, so the standard Duffin lifting is not available, as will be shown in Section~\ref{sec:helmholtz-failure}.
\subsection*{Analytical Derivation of the One-Dimensional Eigenfunction}

A standard reference for the Fourier transformation of the fractional Laplacian is Proposition 3.3 of Di Nezza, Palatucci, and Valdinoci \cite{DiNezzaPalatucciValdinoci2012}.

\begin{pro}[Fourier symbol of the one-dimensional fractional Laplacian]
Let $0<\alpha<2$ and let $\varphi$ be sufficiently regular, for instance $\varphi\in\mathcal{S}(\R)$. With the Fourier transform convention
\[
\mathcal{F}\varphi(\xi)=\int_{\R}\e^{-\i y\xi}\varphi(y)\,\d y,
\]
Proposition 3.3 of \cite{DiNezzaPalatucciValdinoci2012}, applied with $s=\alpha/2$, gives
\[
\mathcal{F}\!\left[(-\Delta)^{\alpha/2}\varphi\right](\xi)
=|\xi|^\alpha \mathcal{F}\varphi(\xi).
\]
Equivalently,
\[
(-\Delta)^{\alpha/2}\varphi
=\mathcal{F}^{-1}\!\left(|\xi|^\alpha\mathcal{F}\varphi(\xi)\right).
\]
\end{pro}

Applying this proposition to the Fourier basis $u(y)=e^{\i\omega y}$, whose Fourier transform is concentrated at the frequency $\xi=\omega$, gives
\[
(-\Delta)^{\alpha/2}_y e^{\i\omega y}=|\omega|^\alpha e^{\i\omega y}.
\]
For completeness, we give the direct derivation below, although the general idea is the same as in \cite[Proposition 3.3]{DiNezzaPalatucciValdinoci2012}. Thus the one-dimensional eigenvalue equation \eqref{eq:eigenfunction} is solved by plane waves with frequencies satisfying $|\omega|^\alpha=\lambda$. In this direct computation we use the normalization constant from \eqref{eq:calpha-definition}; equivalently,
\begin{equation*}
C_\alpha
=\left(\int_{\R}\frac{1-\cos z}{|z|^{1+\alpha}}\,\d z\right)^{-1}
=\frac{\Gamma(1+\alpha)\sin(\pi\alpha/2)}{\pi}.
\end{equation*}
This is the standard normalization that makes the singular-integral definition have Fourier symbol $|\xi|^\alpha$; see \cite[Section 2]{Kwasnicki-2017} and \cite[Proposition 3.3]{DiNezzaPalatucciValdinoci2012}.

Let $f(y)=e^{i\omega y}$, where $\omega\in\R$ is fixed and $y\in\R$ is the spatial variable. Substituting this into \eqref{eq:eigenfunction} gives
\begin{equation*}
(-\Delta)^{\alpha/2}_y e^{i\omega y}
= C_\alpha\int_{-\infty}^{\infty}
\frac{e^{i\omega y}-e^{i\omega x}}{|y-x|^{1+\alpha}}\,\d x.
\end{equation*}
With the change of variables $v=x-y$, we obtain
\begin{equation*}
(-\Delta)^{\alpha/2}_y e^{i\omega y}
= C_\alpha e^{i\omega y}\int_{-\infty}^{\infty}
\frac{1-e^{i\omega v}}{|v|^{1+\alpha}}\,\d v.
\end{equation*}
Using $e^{i\omega v}=\cos(\omega v)+i\sin(\omega v)$, we may split the integral into real and imaginary parts. And the denominator is even and $\sin(\omega v)$ is odd so the principal value integral of a function is zero. The imaginary part by symmetry. Hence
\begin{equation*}
(-\Delta)^{\alpha/2}_y e^{i\omega y}
= C_\alpha e^{i\omega y}\int_{-\infty}^{\infty}
\frac{1-\cos(\omega v)}{|v|^{1+\alpha}}\,\d v.
\end{equation*}
Near $v=0$, the numerator behaves like $\omega^2 v^2/2$, so the integrand behaves like $|v|^{1-\alpha}$ and is locally integrable for $\alpha\in(0,2)$. Thus the principal value integral agrees with the improper integral. For $\omega\ne0$, set $u=|\omega|v$; the case $\omega=0$ is immediate. Then
\begin{equation*}
\int_{-\infty}^{\infty}\frac{1-\cos(\omega v)}{|v|^{1+\alpha}}\,\d v
=|\omega|^\alpha
\int_{-\infty}^{\infty}\frac{1-\cos u}{|u|^{1+\alpha}}\,\d u.
\end{equation*}
The cancellation with $C_\alpha$ is explicit:
\begin{align*}
C_\alpha
\int_{-\infty}^{\infty}\frac{1-\cos(\omega v)}{|v|^{1+\alpha}}\,\d v
&=
|\omega|^\alpha
\frac{\Gamma(1+\alpha)\sin(\pi\alpha/2)}{\pi}
\int_{-\infty}^{\infty}\frac{1-\cos u}{|u|^{1+\alpha}}\,\d u\\
&=
|\omega|^\alpha
\frac{\Gamma(1+\alpha)\sin(\pi\alpha/2)}{\pi}
\cdot
\frac{\pi}{\Gamma(1+\alpha)\sin(\pi\alpha/2)}\\
&=|\omega|^\alpha.
\end{align*}
Therefore, we obtain
\begin{equation*}
(-\Delta)^{\alpha/2}_y e^{i\omega y}=|\omega|^\alpha e^{i\omega y}.
\end{equation*}
Consequently, the solutions of \eqref{eq:eigenfunction} with $\lambda>0$ are generated by frequencies satisfying $|\omega|^\alpha=\lambda$, that is,
\begin{equation*}
\omega=\pm \lambda^{1/\alpha}.
\end{equation*}
The real eigenspace is therefore spanned by
\begin{equation}
f_\lambda(y)=A\cos\!\left(\lambda^{1/\alpha}y\right)+B\sin\!\left(\lambda^{1/\alpha}y\right).
\label{eq:report3-real-eigenspace}
\end{equation}
The following is the core of the Duffin correspondence.  

\begin{thm}[Duffin Correspondence]\label{thm:duffin_corresponence}
Let $D\subset \R^d$ be a bounded domain. Let $\lambda>0$.
Let $f\colon\R\to\R$ be a non-zero eigenfunction given by \eqref{eq:eigenfunction}. Define
\[
\mathcal{A}_{\x,y}^{\alpha}:=A_{\x}^{\alpha}-(-\Delta_y)^{\alpha/2}.
\]
Set $v(\x,y) = u(\x)f(y)$.  Then $v$ solves
\begin{equation}\label{eq:duffin_correspondence-1}
\mathcal{A}_{\x,y}^{\alpha}v(\x,y) = 0 \quad\mbox{on}\quad D\times\R
\end{equation}
if and only if $u$ solves
\begin{equation}\label{eq:duffin_correspondence-2}
A^{\alpha}_\x u(\x) = \lambda u(\x) \quad\mbox{on}\quad D.
\end{equation}
\end{thm}

\begin{proof}
Suppose $u$ satisfies \eqref{eq:duffin_correspondence-2}.  Then, since $f$ satisfies \eqref{eq:eigenfunction}, we have
\begin{eqnarray*}
\mathcal{A}_{\x,y}^{\alpha}v(\x,y) &=&
\sum_{i=1}^{d} -(-\Delta_{x_i})^{\alpha/2}v(\x,y) - (-\Delta_y)^{\alpha/2}v(\x,y) \\
&=&
f(y)\sum_{i=1}^{d} -(-\Delta_{x_i})^{\alpha/2}u(\x) - u(\x)(-\Delta_y)^{\alpha/2}f(y) \\
&=&
f(y)A^{\alpha}_\x u(\x) -u(\x)\lambda f(y) \\
&=&
f(y)\left[\lambda u(\x) - \lambda u(\x)\right] \\
&=&
0.
\end{eqnarray*}
So, $v$ satisfies \eqref{eq:duffin_correspondence-1}.

Suppose now that $v$ satisfies \eqref{eq:duffin_correspondence-1}.  Then we have by the calculations above that
\begin{eqnarray*}
0 &=& \mathcal{A}_{\x,y}^{\alpha}v(\x,y) \\
&=&
f(y)\left[A^{\alpha}_\x u(\x) -\lambda u(\x)\right]. 
\end{eqnarray*}
Since this identity holds for all $y\in\R$ and $f\not\equiv0$, there exists $y_0$ such that $f(y_0)\neq0$; hence $u$ must satisfy \eqref{eq:duffin_correspondence-2}. 
\end{proof}

In particular, the Duffin correspondence provides us a way to solve the Yukawa fractional Dirichlet problem \eqref{eq:fractional-dirichlet}--\eqref{eq:fractional-dirichlet-boundary} as the Laplace problem
\begin{equation}\label{eq:fractional-laplace}
\mathcal{A}_{\x,y}^{\alpha}v(\x,y) = 0 \quad\mbox{on}\quad D\times\R, \mbox{ and }	
\end{equation} 
\begin{equation}\label{eq:fractional-laplace-boundary}
v(\x,y) = g(\x)f(y) \quad\mbox{on the exterior of}\quad  D\times\R. 
\end{equation}	

Recall the Kakutani connection to the fractional Laplace equation \cite{bogdan2000potential}.  Consider the fractional Laplace Dirichlet problem
$$
A^\alpha_\x u(\x) = 0 \mbox{ on } D
$$
$$
u(\x) = g(\x) \mbox{ on } D^c
$$ 
Let $\X =(X^1,\ldots,X^d)$ be an i.i.d. component symmetric $\alpha$-stable L\'evy process, i.e., each component $X^j$, $j=1,\ldots, d$ have the same law given by the characteristic function
$$
\E\left[\e^{\i\theta X^1_t}\right] =
\e^{-|\theta|^\alpha t}.
$$
Then we have a stochastic representation of the solution $u$:
\begin{equation}\label{eq:fractional-kakutani}
u(\x) = \E^\x\left[g(\X_\tau)\right]	
\end{equation}
where $\tau =\inf\{t>0; \X_t\in D^c\}$.
The connection \eqref{eq:fractional-kakutani} provides a method for solving the fractional Laplace Dirichlet problem stochastically.  One simply simulates a large number of paths $\X(i)$, $i=1,\ldots, N$, of the $\alpha$-stable L\'evy process $\X$ starting from $\x$ and the approximate solution to $u$ is
\begin{equation}\label{eq:mc0}
\hat u(\x) = \frac{1}{N}\sum_{i=1}^N g(\X_{\tau(i)}(i)),
\end{equation}
where $\tau(i)$ is the first time the sample $\X(i)$ enters the domain $D^c$.  

\begin{rem}
To use the formula \eqref{eq:mc0} one needs to simulate whole trajectories with fine time mesh of $\X$, and this is computationally heavy and leads to accumulation of errors.  To overcome the problem we provide the Walk-on-Cubes algorithm for the simulation.  This algorithm is based on the well-known Walk-on-Spheres algorithm for the classical Laplace and Brownian motion case, see Muller \cite{Muller-1956} generalizations of which where studied in \cite{Han-Rasila-Sottinen-2025}.
\end{rem}

%%%%%%%%%%%%%%%%%%%%%%%%%%%%%%%%%%%%%%%%%%%%%%%%%%%%%%%%%%%%%%%%%%%%%%%%%%%%%%%
%%%%%%%%%%%%%%%%%%%%%%%%%%%%%%%%%%%%%%%%%%%%%%%%%%%%%%%%%%%%%%%%%%%%%%%%%%%%%%%
\section{Theoretical Background}

\subsection{Operator Definition and Unified Framework in the Pure Fractional Laplace Regime}
We consider the exterior Dirichlet problem
\begin{equation}\label{eq:unified-framework-pde}
A_{\x}^{\alpha}u(\x)-\lambda u(\x)=0,\qquad \x\in D\subset\R^d,
\end{equation}
with boundary condition
\begin{equation*}
u(\x)=g(\x),\qquad \x\in D^c,
\end{equation*}
where
\begin{equation*}
A_{\x}^{\alpha}=\sum_{i=1}^d-(-\Delta_{x_i})^{\alpha/2},\qquad 0<\alpha<2.
\end{equation*}
The operator is Cartesian separable:
\[
A_{\x}^{\alpha}=\sum_{i=1}^d A_{x_i}^{\alpha},
\]
hence each coordinate is governed by a one-dimensional fractional generator. This additive decomposition reflects the nonisotropic structure and leads to a coordinatewise jump mechanism.
The same decomposition also explains why the Walk-on-Cubes (WoC) method is natural for this model. At each state $\z\in D$, WoC selects the maximal axis-aligned cube $\z+rQ$ with $Q=[-1,1]^d$ and
\[
r=\operatorname{dist}_{\infty}(\z,\partial D).
\]

Because the $d$ coordinates evolve independently and the local geometry is described by the coordinate-wise bounds $|z_i-z_{0,i}|<r$ for $i=1,\ldots,d$, the $L_{\infty}$ metric is the natural metric here. Together with the self-similarity of symmetric $\alpha$-stable motions, this gives an exact rescaling rule from the unit-cube exit law to any local cube, which underlies the WoC construction.
When $\lambda=0$, \eqref{eq:unified-framework-pde} reduces to the pure fractional Laplace equation
\[
A_{\x}^{\alpha}u(\x)=0\quad\text{in }D.
\]

From the probabilistic viewpoint, the associated process $\X_t=(X_t^1,\ldots,X_t^d)$ is a drift-free symmetric pure-jump L\'evy motion with independent components, which describes unbiased anomalous diffusion. In particular, there is no deterministic drift, killing, or damping term. If
\[
\tau=\inf\{t>0:\,\X_t\notin D\},
\]
then the solution admits the exit distribution representation
\begin{equation*}
u(\x)=\E_{\x}\!\left[g(\X_{\tau})\right].
\end{equation*}
This representation underlies the Monte Carlo estimator used in the WoC algorithm.
\subsection{Duffin Extension and Oscillatory Screening for the Fractional Yukawa Regime}
\label{sec:yukawa-duffin}
We now turn to the Yukawa case $\lambda>0$ in \eqref{eq:unified-framework-pde}. Set
\[
\kappa:=\lambda^{1/\alpha},\qquad f_\lambda(w):=\cos(\kappa w).
\]
Since the one-dimensional pointwise eigen relation
\begin{equation*}
(-\Delta_w)^{\alpha/2}f_\lambda(w)=\lambda f_\lambda(w),\qquad w\in\R,
\end{equation*}
holds, we define the lifted field
\[
U(\x,w):=u(\x)f_\lambda(w).
\]
Introduce the potential free lifted operator
\begin{equation*}
\mathcal{A}_{\x,w}^{\alpha}:=A_{\x}^{\alpha}-(-\Delta_w)^{\alpha/2}.
\end{equation*}
Then the Yukawa equation in $D\subset\R^d$ is equivalent to the pure fractional Laplace equation in the augmented space:
\begin{equation*}
\mathcal{A}_{\x,w}^{\alpha}U(\x,w)=0,\qquad (\x,w)\in D\times\R,
\end{equation*}
with exterior data
\begin{equation*}
U(\x,w)=g(\x)\cos(\lambda^{1/\alpha}w),\qquad (\x,w)\in D^c\times\R.
\end{equation*}
Let $\widetilde{\X}_t=(\X_t,W_t)$, where $\X_t$ is the $d$-dimensional nonisotropic symmetric $\alpha$-stable motion generated by $A_{\x}^{\alpha}$ and $W_t$ is an independent one-dimensional symmetric $\alpha$-stable process. With
\[
\tau_D:=\inf\{t>0:\X_t\notin D\},
\]
the Duffin representation becomes
\begin{equation*}
u(\x)=\E_{(\x,0)}\!\left[g(\X_{\tau_D})\cos\!\big(\lambda^{1/\alpha}W_{\tau_D}\big)\right].
\end{equation*}
Hence the Monte Carlo estimator uses i.i.d. samples of the oscillatory payoff
\[
\xi_k:=g(\X^{(k)}_{\tau_D})\cos\!\big(\lambda^{1/\alpha}W^{(k)}_{\tau_D}\big),
\]
and, under $\E|\xi_1|<\infty$, the strong law yields
\[
\frac1N\sum_{k=1}^N\xi_k\xrightarrow[N\to\infty]{\mathrm{a.s.}}u(\x).
\]
The screening mechanism is oscillatory rather than dissipative. For fixed $t>0$,
\[
\E\!\left[\cos\!\big(\lambda^{1/\alpha}W_t\big)\right]
=\Re\,\E\!\left[\exp\!\big(i\lambda^{1/\alpha}W_t\big)\right]
=e^{-\lambda t},
\]
so phase mixing leads to cancellation at the expectation level. From a pathwise viewpoint, trajectories started deep inside $D$ typically require many WoC updates before they exit, and the heavy-tailed jump structure gives a broad right tail for $\tau_D$; accordingly, the scale of $W_{\tau_D}$ (of order $\tau_D^{1/\alpha}$) becomes large. The factor $\cos(\lambda^{1/\alpha}W_{\tau_D})\in[-1,1]$ then oscillates strongly across samples, which leads to substantial cancellation in empirical averages. At the level of the computed solution, this appears as an interior region that is nearly flat and close to zero, together with a sharp boundary layer where $\tau_D$ is small and the boundary data are less screened.
\subsection{Failure of the Duffin Extension in the Helmholtz Regime and a Feynman--Kac Reconstruction}
\label{sec:helmholtz-failure}
The Helmholtz regime corresponds to $\lambda<0$ in \eqref{eq:unified-framework-pde}. Writing
\[
\kappa:=-\lambda>0,
\]
the equation becomes
\[
A_{\x}^{\alpha}u(\x)+\kappa u(\x)=0\quad\text{in }D.
\]
A direct reuse of the Duffin space extension paradigm leads to a genuine functional analytic obstruction. Formally, one would seek an auxiliary profile $f$ solving
\[
(-\Delta_w)^{\alpha/2}f(w)=-\kappa f(w),
\]
whose natural candidates are exponentially growing/decaying modes. However, for the fractional operator with heavy-tailed kernel,
\begin{equation}\label{eq:helmholtz-kernel-pv}
(-\Delta_w)^{\alpha/2}f(w)=c_\alpha\int_{\R}\frac{f(w)-f(w+z)}{|z|^{1+\alpha}}\,\d z,
\end{equation}
the tail $|z|^{-(1+\alpha)}$ is only polynomially decaying. If $f(w)=e^{\beta w}$ with $\beta\neq0$, then
\[
\frac{f(w)-f(w+z)}{|z|^{1+\alpha}}
=e^{\beta w}\frac{1-e^{\beta z}}{|z|^{1+\alpha}},
\]
then for $\beta>0$ the tail $z\to+\infty$ grows exponentially, while for $\beta<0$ the tail $z\to-\infty$ grows exponentially. Hence the integral in \eqref{eq:helmholtz-kernel-pv} diverges at infinity. Therefore the naive lifted space Duffin correspondence is not well-posed for $\lambda<0$, and a simple spatial augmentation is not available.

The appropriate replacement is a time enhanced Feynman--Kac reconstruction. Let $\X_t=(X_t^1,\ldots,X_t^d)$ be the nonisotropic symmetric $\alpha$-stable process generated by $A_{\x}^{\alpha}$, and let
\[
\tau_D:=\inf\{t>0:\X_t\notin D\}.
\]
Then the Helmholtz solution is represented by
\begin{equation*}
u(\x)=\E_{\x}\!\left[e^{-\lambda\tau_D}g(\X_{\tau_D})\right]
=\E_{\x}\!\left[e^{\kappa\tau_D}g(\X_{\tau_D})\right].
\end{equation*}
Thus the Helmholtz weight grows exponentially rather than decays. This is the source of the instability in Monte Carlo simulation.
In WoC form, each local step from a cube of radius $r$ inherits the space--time self-similarity of the $\alpha$-stable dynamics:
\begin{equation*}
\Delta\tau=r^{\alpha}\tau_{\mathrm{std}},
\end{equation*}
where $\tau_{\mathrm{std}}$ is the first exit time from the unit cube. Along one trajectory,
\begin{equation*}
\tau_{\mathrm{total}}=\sum_{j=0}^{N-1} r_j^{\alpha}\tau_{\mathrm{std}}^{(j)}.
\end{equation*}
Here $\tau_{\mathrm{total}}$ is a discrete approximation of the continuous stopping time $\tau_D$, obtained from the space--time self-similarity of the $\alpha$-stable process. Under appropriate regularity conditions on the boundary and the fractional Poisson kernel, and for a consistent WoC refinement ($\epsilon\to0$, $N\to\infty$), one expects
\[
\tau_{\mathrm{total}}^{(\epsilon,N)}\xrightarrow[\epsilon\to0,\,N\to\infty]{\mathcal D}\tau_D.
\]
Hence the path payoff is reweighted as
\begin{equation*}
\Xi=g(\X_{\tau_D})\exp\!\big(-\lambda\tau_{\mathrm{total}}\big),
\end{equation*}
and the Monte Carlo estimator is the empirical mean of i.i.d. copies of $\Xi$.
Let $\lambda_1(D)>0$ denote the principal Dirichlet eigenvalue of $-A_{\x}^{\alpha}$ on $D$. The single particle Feynman--Kac representation is meaningful only in the subcritical gauge regime \cite{chen2002gaugeability}
\[
-\lambda<\lambda_1(D),
\]
which ensures finiteness of the relevant exponential moments. When $-\lambda\ge\lambda_1(D)$, the variance, or even the expectation, can blow up, and one typically needs branching random walk mechanisms rather than a pure single trajectory reweighting scheme.
\subsection{Existence of the Principal Eigenvalue and Its WoC Simulation}

We now show that the principal Dirichlet eigenvalue $\lambda_1(D)$ exists and explain why the Walk-on-Cubes exit time simulation can be used to estimate it. This is the quantity that controls the admissible range of the Helmholtz parameter in the Feynman--Kac reconstruction.

\begin{dfn}[Rectilinear stable process and killed semigroup]
Let $\X=(X^1,\ldots,X^d)$ be the rectilinear symmetric $\alpha$-stable L\'evy process, i.e., the coordinates are independent one-dimensional symmetric $\alpha$-stable L\'evy processes. Equivalently, by the L\'evy--Khintchine representation \cite[Chapter 8]{Sato-2013},
\begin{equation*}
\E_0\!\left[e^{\i\langle \xi,X_t\rangle}\right]
=\exp\!\left(-t\Psi(\xi)\right),
\qquad
\Psi(\xi)=\sum_{i=1}^d|\xi_i|^\alpha .
\end{equation*}
The generator is $A^\alpha=-L^\alpha$, where
\begin{equation*}
L^\alpha:=\sum_{i=1}^d(-\Delta_{x_i})^{\alpha/2}.
\end{equation*}
For a bounded open set $D\subset\R^d$, define
\begin{equation*}
\tau_D:=\inf\{t>0:X_t\notin D\},
\end{equation*}
and the killed semigroup on $D$ by
\begin{equation*}
T_t^D f(x):=\E_x\!\left[f(X_t)\mathbf{1}_{\{\tau_D>t\}}\right].
\end{equation*}
The terminology ``rectilinear stable process'' and the Dirichlet transition density theory used below follow Chen, Hu, and Zhao \cite{ChenHuZhao2025}.
\end{dfn}

The corresponding Dirichlet form is the coordinate axis form
\begin{equation}\label{eq:rectilinear-dirichlet-form}
\mathcal E_D(f,f)
=\frac{C_\alpha}{2}\sum_{i=1}^d
\int_{\R^d}\int_{\R}
\frac{\left(f(x+he_i)-f(x)\right)^2}{|h|^{1+\alpha}}\,\d h\,\d x,
\qquad f=0\ \hbox{on }D^c,
\end{equation}
where $e_i$ is the $i$th coordinate vector. Formula \eqref{eq:rectilinear-dirichlet-form} is the quadratic form version of the singular integral generator \eqref{eq:flaplace}; see also the equivalent definitions of the fractional Laplacian in \cite{Kwasnicki-2017}.

\begin{dfn}[Principal Dirichlet eigenvalue]
The principal Dirichlet eigenvalue of $L^\alpha=-A^\alpha$ on $D$ is
\begin{equation*}
\lambda_1(D):=
\inf\left\{\mathcal E_D(f,f): f\in\mathcal F_D,\ \|f\|_{L^2(D)}=1\right\},
\end{equation*}
where $\mathcal F_D$ is the closure of smooth functions compactly supported in $D$ under the energy norm associated with \eqref{eq:rectilinear-dirichlet-form}. This is the Rayleigh--Ritz characterization of the bottom of the Dirichlet spectrum; compare the spectral treatment of fractional Dirichlet eigenvalues in Kwa\'snicki \cite{Kwasnicki-2012}.
\end{dfn}

\begin{pro}[Existence and survival tail characterization of $\lambda_1(D)$]\label{pro:lambda1-survival-tail}
Assume that $D$ is bounded and belongs to the geometric class for which the killed rectilinear stable process is irreducible and has a strictly positive Dirichlet transition density in the sense of \cite{ChenHuZhao2025}. Then $T_t^D$ is a compact, self-adjoint, positivity improving operator on $L^2(D)$ for every $t>0$. Consequently, there are eigenvalues
\begin{equation}\label{eq:discrete-dirichlet-spectrum}
0<\lambda_1(D)<\lambda_2(D)\le\lambda_3(D)\le\ldots,\qquad
\lambda_n(D)\to\infty,
\end{equation}
and an orthonormal basis $\{\phi_n\}_{n\ge1}$ of $L^2(D)$ such that
\begin{equation}\label{eq:killed-semigroup-eigen-exp}
T_t^D\phi_n=e^{-t\lambda_n(D)}\phi_n.
\end{equation}
Moreover, $\lambda_1(D)$ is simple and $\phi_1$ can be chosen strictly positive on $D$. For every point $x$ for which the transition density expansion is valid,
\begin{equation}\label{eq:survival-spectral-expansion}
S_x(t):=\P_x(\tau_D>t)
=\sum_{n=1}^{\infty}e^{-t\lambda_n(D)}\phi_n(x)\int_D\phi_n(y)\,\d y,
\end{equation}
and therefore
\begin{equation}\label{eq:survival-principal-tail}
\P_x(\tau_D>t)=C_D(x)e^{-\lambda_1(D)t}(1+o(1)),
\qquad C_D(x)>0.
\end{equation}
In particular,
\begin{equation}\label{eq:lambda1-log-survival-limit}
\lambda_1(D)
=-\lim_{t\to\infty}\frac1t\log \P_x(\tau_D>t).
\end{equation}
\end{pro}

\begin{proof}
The transition density $p_D(t,x,y)$ and its strict positivity under the stated irreducibility assumptions are supplied by \cite{ChenHuZhao2025}. Thus
\begin{equation*}
T_t^D f(x)=\int_D p_D(t,x,y)f(y)\,\d y.
\end{equation*}
Let $p(t,x,y)$ denote the free transition density of the rectilinear stable process on the whole space $\R^d$, i.e. before imposing the killing at $\partial D$.  Because the coordinates are independent,
\begin{equation*}
p(t,x,y)=\prod_{i=1}^d p_t^{(1D)}(y_i-x_i),
\end{equation*}
where the one dimensional symmetric $\alpha$-stable density is
\begin{equation*}
p_t^{(1D)}(z)
=\frac{1}{2\pi}\int_{\R}e^{\i\xi z}e^{-t|\xi|^\alpha}\,\d\xi
=\frac{1}{\pi}\int_0^\infty e^{-t\xi^\alpha}\cos(\xi z)\,\d\xi.
\end{equation*}
Since $p_D(t,x,y)\le p(t,x,y)$, the stable scaling in \cite[Chapter 3]{Bertoin-1996} gives
\begin{equation*}
p(2t,x,x)\le C t^{-d/\alpha}.
\end{equation*}
Using symmetry and the semigroup property,
\begin{equation*}
\iint_{D\times D}p_D(t,x,y)^2\,\d x\,\d y
=\int_D p_D(2t,x,x)\,\d x
\le |D|Ct^{-d/\alpha}<\infty.
\end{equation*}
Thus $T_t^D$ is Hilbert--Schmidt, hence compact on $L^2(D)$. The spectral theorem for compact self-adjoint operators \cite[Chapter VI]{Riesz-SzNagy-1955} gives \eqref{eq:discrete-dirichlet-spectrum}--\eqref{eq:killed-semigroup-eigen-exp}. Because $T_t^D$ is positivity improving, the Jentzsch--Krein--Rutman theorem for compact positive operators \cite[Chapter V]{Schaefer1974} gives simplicity and positivity of the leading eigenfunction.

Finally,
\begin{equation*}
S_x(t)=T_t^D1(x)=\int_Dp_D(t,x,y)\,\d y.
\end{equation*}
Expanding $1$ in the eigenbasis gives \eqref{eq:survival-spectral-expansion}. Since $\phi_1>0$, the coefficient
\begin{equation*}
C_D(x):=\phi_1(x)\int_D\phi_1(y)\,\d y
\end{equation*}
is strictly positive. The term with $n=1$ dominates the expansion as $t\to\infty$, which yields \eqref{eq:survival-principal-tail} and the logarithmic limit \eqref{eq:lambda1-log-survival-limit}.
\end{proof}

The probabilistic meaning of \eqref{eq:lambda1-log-survival-limit} is that the principal eigenvalue is the exponential decay rate of the survival probability. This is exactly the quantity that can be observed from simulated exit times; no eigenfunction needs to be computed.

\begin{pro}[WoC estimator for the principal eigenvalue]\label{pro:woc-lambda1-estimator}
Let $Q=[-1,1]^d$, and suppose the precomputed jump pool contains independent samples of the unit-cube exit pair
\begin{equation*}
\left(Y_{\sigma_Q},\sigma_Q\right),
\qquad
\sigma_Q:=\inf\{t>0:Y_t\notin Q\},
\end{equation*}
where $Y_t$ is the rectilinear stable process started at the origin. If the current WoC position is $z\in D$ and
\begin{equation*}
r(z):=\operatorname{dist}_{\infty}(z,D^c),
\end{equation*}
then the stable self-similarity and the strong Markov property imply the exact local rescaling law
\begin{equation}\label{eq:woc-local-rescaling}
\left(X_{\tau_{z+rQ}},\tau_{z+rQ}\right)
\stackrel{d}{=}
\left(z+rY_{\sigma_Q},\,r^\alpha\sigma_Q\right),
\qquad X_0=z.
\end{equation}
Here $\tau_{z+rQ}:=\inf\{t>0:X_t\notin z+rQ\}$ is the first exit time from the local cube.
Consequently, a WoC path has accumulated time
\begin{equation*}
\tau_{\epsilon}
=\sum_{j=0}^{N_\epsilon-1} r_j^\alpha\sigma_Q^{(j)},
\end{equation*}
where $r_j=r(Z_j)$ and the path stops when $r_j\le\epsilon$ or when the sampled exit point leaves $D$. If the unit-cube exit law is sampled exactly and $\epsilon\downarrow0$, then $\tau_{\epsilon}$ converges in distribution to $\tau_D$. Hence independent WoC paths give the empirical survival curve
\begin{equation*}
\widehat S_{x,N}^{(\epsilon)}(t)
=\frac1N\sum_{k=1}^N
\mathbf{1}_{\{\tau_{\epsilon}^{(k)}>t\}},
\end{equation*}
and, for each fixed fitting grid, the strong law of large numbers gives
\begin{equation*}
\widehat S_{x,N}^{(\epsilon)}(t)
\xrightarrow[N\to\infty]{\mathrm{a.s.}}
\P_x(\tau_{\epsilon}>t)
\xrightarrow[\epsilon\downarrow0]{}
\P_x(\tau_D>t).
\end{equation*}
Combining this with \eqref{eq:lambda1-log-survival-limit}, the slope of the late-time log survival curve estimates $\lambda_1(D)$:
\begin{equation*}
\log \widehat S_{x,N}^{(\epsilon)}(t)
\approx -\lambda_1(D)t+b_x.
\end{equation*}
\end{pro}

The rescaling \eqref{eq:woc-local-rescaling} is the reason WoC is appropriate here. The local domain used by the algorithm is an $L_\infty$ cube, and the rectilinear process has independent coordinate jumps and the scaling $X_{ct}\stackrel{d}{=}c^{1/\alpha}X_t$; see the stable process references \cite{Bertoin-1996,Sato-2013}. The construction is the cube analogue of the walk-on-spheres principle of Muller \cite{Muller-1956}, and it plays the same role for fractional exit problems as the isotropic fractional walk-on-spheres algorithm in \cite{Kyprianou-Osojnik-Shardlow-2018}.

For implementation, choose late-time grid points $t_1<\cdots<t_m$ with nonzero empirical survival counts, and compute
\begin{equation*}
(\widehat\lambda_1,\widehat b_x)
=\arg\min_{\ell,b}
\sum_{q=1}^m
\left(\log\widehat S_{x,N}^{(\epsilon)}(t_q)+\ell t_q-b\right)^2.
\end{equation*}
The first component of the minimizer is the estimate $\widehat\lambda_1$. The fitting window should be late enough that the first eigenmode dominates in \eqref{eq:survival-spectral-expansion}, but not so late that $\widehat S$ is dominated by a small number of surviving samples.

Finally, this eigenvalue estimate is not merely diagnostic. If $\lambda<0$ and $\kappa=-\lambda>0$, the Helmholtz Feynman--Kac weight is $\e^{\kappa\tau_D}$. From the tail identity
\begin{equation*}
\E_x\left[\e^{a\tau_D}\right]
=1+a\int_0^\infty e^{at}\P_x(\tau_D>t)\,\d t,\qquad a>0,
\end{equation*}
together with \eqref{eq:survival-principal-tail}, one obtains
\begin{equation*}
\E_x[e^{a\tau_D}]<\infty
\quad\Longleftrightarrow\quad
a<\lambda_1(D).
\end{equation*}
This is the spectral form of the gaugeability threshold for Schr\"odinger-type perturbations of Markov processes \cite{bogdan2000potential,chen2002gaugeability}. Thus the expectation in the Helmholtz representation is finite when
\begin{equation*}
\kappa=-\lambda<\lambda_1(D),
\end{equation*}
and, for bounded boundary data, the Monte Carlo variance is finite under the stronger $L^2$ condition
\begin{equation*}
2\kappa=2|\lambda|<\lambda_1(D).
\end{equation*}
This is why the WoC survival profiler can be used before the Helmholtz solve: it estimates the largest safe exponential growth rate allowed by the domain.
%%%%%%%%%%%%%%%%%%%%%%%%%%%%%%%%%%%%%%%%%%%%%%%%%%%%%%%%%%%%%%%%%%%%%%%%%%%%%%%
%%%%%%%%%%%%%%%%%%%%%%%%%%%%%%%%%%%%%%%%%%%%%%%%%%%%%%%%%%%%%%%%%%%%%%%%%%%%%%%
\section{Walk-on-Cubes}

Based on the regime analysis in Section 3, we formulate a unified Walk-on-Cubes (WoC) Monte Carlo solver for the Dirichlet problem $A_{\x}^{\alpha}u-\lambda u=0$ on $D\subset\R^d$. The implementation uses a common cube-exit kernel together with regime-dependent weighting: for $\lambda=0$, coordinatewise $\alpha$-stable pure-jump scaling determines the spatial transport; for $\lambda>0$, a Duffin extension adds an auxiliary coordinate $w$ and the boundary payoff includes the oscillatory factor; for $\lambda<0$, a Feynman--Kac time enhancement accumulates survival time through $\Delta\tau=r^{\alpha}\tau_{\mathrm{std}}$ and applies the exponential reweighting $\exp(-\lambda\tau_{\mathrm{total}})$ under the subcritical gauge constraint. In this way, the Laplace, Yukawa, and Helmholtz cases are treated within a single implementation framework that is well suited to parallel computation on modern GPUs.
\subsection{Walk-on-Cubes Algorithm}

We consider the numerical solution of the nonisotropic fractional-order partial differential equation with a potential term:
\begin{equation}
    A_{\x}^\alpha u(\x) - \lambda u(\x) = 0, \quad \x \in D \subset \R^d.
\end{equation}
subject to the Dirichlet exterior boundary condition $u(\x) = g(\x)$ for $\x \in D^c$. 
Here, $\alpha \in (0, 2)$, and $A_{\x}^\alpha$ is the nonisotropic fractional Laplacian, defined as the sum of one-dimensional fractional operators:
\begin{equation}
    A_{\x}^\alpha = \sum_{i=1}^d -(-\Delta_{x_i})^{\alpha/2}.
\end{equation}

\subsection{Mathematical Formulation of the Algorithm  }

To implement this problem efficiently on a parallel GPU architecture, we separate the regime-dependent mathematical input from the online path update used in the WoC solver.
The role of the potential parameter $\lambda$ depends on the regime. For $\lambda>0$, the Duffin extension introduces an auxiliary coordinate $w$ and the lifted field satisfies
\begin{equation}
    \mathcal{A}_{\x,w}^{\alpha}U(\x,w)=0,\quad (\x,w)\in D\times\R,
\end{equation}
where $\mathcal{A}_{\x,w}^{\alpha}:=A_{\x}^{\alpha}-(-\Delta_w)^{\alpha/2}$ and $U(\x,w)=u(\x)f_\lambda(w)$. For $\lambda<0$, no spatial lifting is used; instead, the solver accumulates survival time and applies the exponential reweighting $\exp(-\lambda\tau_{\mathrm{total}})$. For $\lambda=0$, neither correction is needed.
The main implementation issue is that a direct simulation of fractional microsteps inside each CUDA thread produces severe warp divergence: the number of microsteps required before leaving a local cube is random and can vary substantially from one thread to another. To reduce this cost, we move this variable length part of the computation to an offline preprocessing stage and precompute a \textbf{Jump Pool} $\mathcal{P}$ of standardized exits from the unit $L_\infty$ cube $Q=[-1,1]^d$. The online solver then uses only a distance query, a pool lookup, and a local rescaling at each WoC update.
The infinitesimal generator of the nonisotropic operator $A_{\x}^\alpha$ corresponds to a $d$-dimensional L\'evy process with \textbf{independent} symmetric $\alpha$-stable components. The standard exit vectors are sampled empirically by simulating discrete paths from the origin. By the self-similarity of the $\alpha$-stable process ($\X_{ct} \stackrel{d}{=} c^{1/\alpha} \X_t$), the position update at step $n$ for a time step $\Delta t$ is given by
\begin{equation}
    \X^{(n+1)} = \X^{(n)} + (\Delta t)^{1/\alpha} \bxi^{(n)},
\end{equation}
where $\bxi^{(n)} = (\xi_1^{(n)}, \dots, \xi_d^{(n)})$ consists of \textbf{independent} standard symmetric $\alpha$-stable random variables, in agreement with the nonisotropic sum formulation.

To generate each component $\xi_i$, we use the \textbf{Chambers-Mallows-Stuck (CMS) method} \cite{Chambers-Mallows-Stuck-1976}. We sample a uniform angle $V$ and an independent standard exponential variable $W$:
\begin{equation}
    V \sim \mathcal{U}\left(-\frac{\pi}{2}, \frac{\pi}{2}\right), \quad W \sim \text{Exp}(1).
\end{equation}
The exact $\alpha$-stable increment is computed by
\begin{equation}
    \xi_i = \frac{\sin(\alpha V)}{(\cos V)^{1/\alpha}} \left( \frac{\cos(V - \alpha V)}{W} \right)^{\frac{1-\alpha}{\alpha}}.
\end{equation}
A path halts at step $N$ when it breaches the boundary ($\X^{(N)} \notin Q$). This terminal state yields one standard space--time sample $(\Delta \X_{\mathrm{std}},\tau_{\mathrm{std}})$, and the resulting pool is
\begin{equation}
    \mathcal{P} = \left\{ \left(\Delta \X_{\mathrm{std}}^{(k)},\tau_{\mathrm{std}}^{(k)}\right) \right\}_{k=1}^{N_{\mathrm{pool}}}.
\end{equation}
After the jump pool has been constructed, the physical domain $D$ is discretized on a uniform grid, and independent \texttt{curandState} objects are assigned to the grid points so that the random number streams are uncorrelated. The online WoC solver is summarized in Algorithm~\ref{alg:unified-woc}. In the algorithm, $\Z^{(n)}$ denotes the full path state; its spatial component is $\x^{(n)}$. Thus $\Z^{(n)}=(\x^{(n)},w^{(n)})$ in the Yukawa regime $\lambda>0$, while $\Z^{(n)}=\x^{(n)}$ in the Laplace and Helmholtz regimes.
\begin{algorithm}[H]
\caption{A unified GPU-optimized Walk-on-Cubes scheme}
\label{alg:unified-woc}
\begin{algorithmic}[1]
\Require Domain $D$, boundary data $g$, parameter $\lambda$, jump pool $\mathcal{P}$, tolerance $\epsilon$, number of shots $N_{shots}$
\ForAll{grid points $\x\in D$ in parallel}
    \For{$m=1,\dots,N_{shots}$}
        \State Initialize the path state $\Z^{(0)}$ at $\x$
        \If{$\lambda>0$}
            \State Augment the state with $w^{(0)}=0$
        \EndIf
        \If{$\lambda<0$}
            \State Initialize $\tau_{\mathrm{total}} \gets 0$
        \EndIf
        \While{the spatial component of $\Z^{(n)}$ lies in $D$}
            \State $r \gets \operatorname{dist}_{\infty}(\x^{(n)},\partial D)$
            \If{$r<\epsilon$}
                \State \textbf{break}
            \EndIf
            \State Sample $(\Delta \X_{\mathrm{std}},\tau_{\mathrm{std}})$ uniformly from $\mathcal{P}$
            \State Update the active spatial coordinates by $\x^{(n+1)} \gets \x^{(n)} + r\,\Delta \X_{\mathrm{std}}$
            \If{$\lambda<0$}
                \State $\tau_{\mathrm{total}} \gets \tau_{\mathrm{total}} + r^{\alpha}\tau_{\mathrm{std}}$
            \EndIf
            \If{the spatial component $\x^{(n+1)}$ lies in $D^c$}
                \State \textbf{break}
            \EndIf
        \EndWhile
        \State Evaluate the terminal payoff by \eqref{eq:woc-payoff}
    \EndFor
    \State Set $\hat u(\x)$ equal to the arithmetic mean of the $N_{shots}$ path payoffs
\EndFor
\end{algorithmic}
\end{algorithm}
The regime-dependent terminal payoff is
\begin{equation}\label{eq:woc-payoff}
\text{Payoff} =
\begin{cases}
    g(\x_{\mathrm{exit}}), & \lambda = 0,\\
    g(\x_{\mathrm{exit}}) \cdot f_\lambda(w_{\mathrm{exit}}), & \lambda > 0,\\
    g(\x_{\mathrm{exit}}) \cdot \exp\!\big(-\lambda\tau_{\mathrm{total}}\big), & \lambda < 0.
\end{cases}
\end{equation}
In the Helmholtz case, this uses the Feynman--Kac exponential weight, which is amplifying because $\lambda<0$, and it is valid in the subcritical gauge regime \cite{chen2002gaugeability} so that the relevant exponential moments remain finite.
\section{Numerical Considerations and Examples}
\subsection*{Variance Blow-up from Boundary Conditions}
Care is required when defining the exterior Dirichlet boundary condition $g(\x)$. If $g$ grows too rapidly with distance, the numerical solution may exhibit large irregular fluctuations. This is caused by the heavy-tailed nature of the fractional process. The domain exterior Poisson kernel decays according to a power law ($\propto |\x|^{-d-\alpha}$) \cite{bogdan1997boundary}, so particles still have a non-negligible probability of making very large jumps into the far exterior. Coupling these heavy-tailed exit locations with an unbounded rapidly growing $g(\x)$ can make the empirical variance of the Monte Carlo estimator infinite or extremely large, thereby destroying the smoothness of the computed expectation surface.
%\subsection{Examples}

To validate the robustness of the generalized solver, we first report two Yukawa Green function manufactured benchmarks, then a one-dimensional Helmholtz benchmark with an exact Fourier mode solution, and finally four additional test cases spanning multiply connected, non-convex, and curvilinear domains in both 2D and 3D. 

\subsection*{Green Function Manufactured Benchmark}
The Green function benchmarks use an exact solution generated from the full-space resolvent of the positive operator
\[
L^\alpha:=-A^\alpha=\sum_{i=1}^d(-\Delta_{x_i})^{\alpha/2}.
\]
For $\lambda>0$, using the Fourier multiplier characterization of the fractional Laplacian \cite[Proposition~3.3]{DiNezzaPalatucciValdinoci2012}, the full-space resolvent kernel is
\[
G_{\lambda,\alpha}(\x)
=\mathcal F^{-1}\!\left[
\frac{1}{\lambda+\sum_{i=1}^d|\xi_i|^\alpha}
\right](\x).
\]
Here $\mathcal F^{-1}$ denotes the inverse Fourier transform.
Equivalently, using the $q$-resolvent/potential operator representation for L\'evy semigroups \cite[Chapter~8, Section~41]{Sato-2013} and the rectilinear transition density product formula \cite[(1.8)]{ChenHuZhao2025}, if $p_t$ denotes the transition density of the rectilinear stable process generated by $A^\alpha$, then
\[
G_{\lambda,\alpha}(\x)=\int_0^\infty e^{-\lambda t}p_t(\x)\,\d t.
\]
For the broader anisotropic coordinate-sum setting, potential kernels and Green function estimates are treated by Bogdan and Sztonyk \cite[Sections~1--3]{Bogdan-Sztonyk-2007}.

By construction,
\[
(L^\alpha+\lambda)G_{\lambda,\alpha}=\delta_0
\quad\mbox{in }\mathcal D'(\R^d),
\]
because the Fourier transform of the left-hand side is
\[
\left(\lambda+\sum_{i=1}^d|\xi_i|^\alpha\right)
\frac{1}{\lambda+\sum_{i=1}^d|\xi_i|^\alpha}=1.
\]
Now place the pole at a point $\mathbf{a}\notin \overline D$ and set
\[
u_{\mathrm{exact}}(\x):=G_{\lambda,\alpha}(\x-\mathbf{a}).
\]
Then $(L^\alpha+\lambda)u_{\mathrm{exact}}=0$ in $D$, or equivalently
\[
A^\alpha u_{\mathrm{exact}}=\lambda u_{\mathrm{exact}}\quad\mbox{in }D.
\]
Thus this gives an exact Yukawa regime benchmark for \eqref{eq:fractional-dirichlet} with $\lambda>0$.

The nonlocal boundary condition is imposed on the whole exterior $D^c$, not only on $\partial D$; this is the standard exterior Dirichlet formulation for fractional Laplace problems \cite{Claus-Warma-2020,RosOton-Serra-2014}. Therefore we prescribe
\[
g(\x)=u_{\mathrm{exact}}(\x)=G_{\lambda,\alpha}(\x-\mathbf{a}),\qquad \x\in D^c.
\]

For one-dimensional tests the exact value can be evaluated by the Fourier integral
\[
G_{\lambda,\alpha}(\x-\mathbf{a})
=\frac1\pi\int_0^\infty
\frac{\cos(\xi(x_1-a_1))}{\lambda+\xi^\alpha}\,\d\xi,
\qquad \x=(x_1),\quad \mathbf{a}=(a_1).
\]
For the coordinate-sum operator in several dimensions, the time resolvent formula is convenient because the free transition density factorizes coordinatewise:
\[
p_t(\x)=\prod_{i=1}^d p_t^{(1D)}(x_i).
\]
Here the one-dimensional factor is the symmetric $\alpha$-stable transition density defined by the Fourier representation
\[
p_t^{(1D)}(z)
=\frac{1}{2\pi}\int_{\R} e^{i\xi z}e^{-t|\xi|^\alpha}\,\d\xi
=\frac{1}{\pi}\int_0^\infty e^{-t\xi^\alpha}\cos(\xi z)\,\d\xi.
\]
	In practice, the pole is kept a positive distance outside $D$ so that the singularity of $G_{\lambda,\alpha}$ is never sampled inside the computational domain. Moderate values such as $\alpha>1$ and $\lambda$ of order $10^{-1}$--$1$ keep the Green function well behaved and make the benchmark a clean test of the WoC implementation.

\subsubsection*{Case 11: One-Dimensional Yukawa Green Comparison}
The first Green function comparison case validates the positive Yukawa branch $\lambda>0$ for the one-dimensional operator
\[
A^\alpha=-(-\Delta)^{\alpha/2}.
\]
The computational domain is $D=(-1,1)$, with $\x=(x_1)$, and the benchmark equation is
\[
A^\alpha u=\lambda u,
\qquad\mbox{equivalently}\qquad
\big((-\Delta)^{\alpha/2}+\lambda\big)u=0
\quad\mbox{in }D.
\]
The exact solution is generated by a one-dimensional fractional Yukawa Green function with a pole placed outside the interval:
\[
\begin{aligned}
u_{\mathrm{exact}}(\x)
&=G_{\lambda,\alpha}(\x-\mathbf{a})
=G_{\lambda,\alpha}(x_1-a_1),\\
G_{\lambda,\alpha}(r)
&=\frac1\pi\int_0^\infty
\frac{\cos(\xi r)}{\lambda+\xi^\alpha}\,\d\xi,
\qquad \mathbf{a}=(a_1),\quad a_1=2.5.
\end{aligned}
\]
Since $\mathbf{a}\notin[-1,1]$, the singularity never enters the computational domain and the distributional resolvent identity gives
\[
\big((-\Delta_{x_1})^{\alpha/2}+\lambda\big)G_{\lambda,\alpha}(x_1-a_1)=0,
\qquad \x=(x_1)\in(-1,1).
\]
The exterior Dirichlet data are therefore fixed by the same full-space manufactured solution,
\[
g(\x)=G_{\lambda,\alpha}(x_1-2.5),
\qquad \x=(x_1)\in D^c.
\]

\begin{table}[H]
\centering
\begin{tabular}{ll}
\hline
Parameter & Value \\
\hline
Domain & $(-1,1)$ \\
Fractional order & $\alpha=1.5$ \\
Yukawa parameter & $\lambda=0.1$ \\
Auxiliary frequency & $\omega=\lambda^{1/\alpha}\approx0.215443$ \\
Pole location & $\mathbf{a}=(2.5)$ \\
Iteration shots & $1{,}000{,}000$ \\
Jump pool size & $1{,}000{,}000$ \\
Grid size & $64$ \\
Random seed & $20260505$ \\
Time step & $10^{-4}$ \\
Exit tolerance & $10^{-5}$ \\
Maximum steps & $20000$ \\
Quadrature cutoff and panels & $300.0,\ 6000$ \\
Green function table size and range & $2049,\ |r|\le40$ \\
\hline
\end{tabular}
\caption{Configuration for the one-dimensional Yukawa Green comparison case.}
\end{table}

\begin{figure}[H]
\centering
\includegraphics[width=0.86\textwidth]{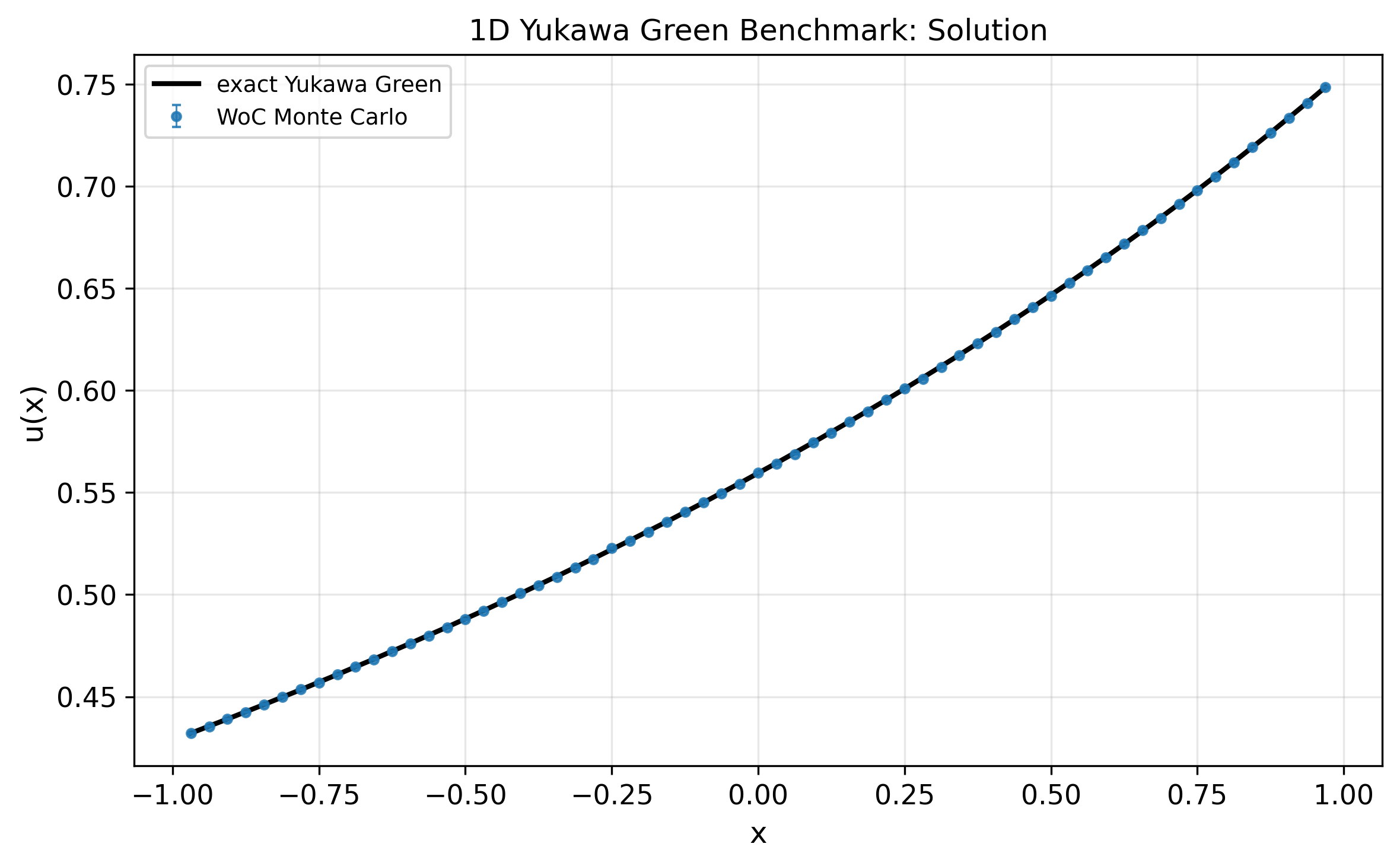}
\caption{One-dimensional Yukawa Green comparison. The WoC Monte Carlo estimate is plotted against the exact Green function solution.}
\label{fig:yukawa-green-1d-solution}
\end{figure}

\begin{figure}[H]
\centering
\includegraphics[width=0.86\textwidth]{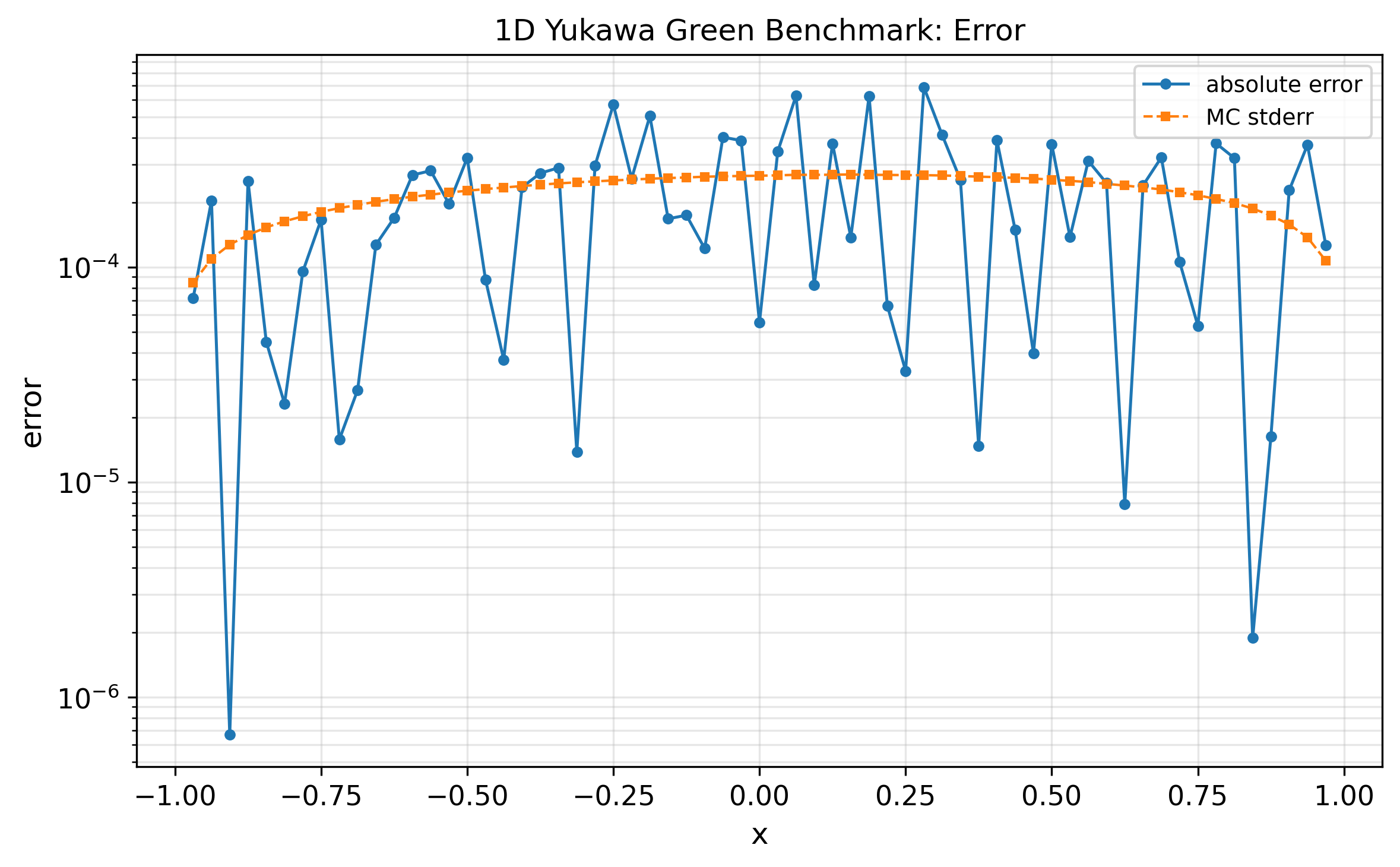}
\caption{Pointwise error for the one-dimensional Yukawa Green comparison. The absolute error is shown together with the Monte Carlo standard error.}
\label{fig:yukawa-green-1d-error}
\end{figure}

Figures~\ref{fig:yukawa-green-1d-solution}--\ref{fig:yukawa-green-1d-error} show that the Monte Carlo curve follows the analytic Yukawa Green profile, and the pointwise error remains on the same scale as the sampling uncertainty.

\subsubsection*{Case 12: Two-Dimensional Separable Yukawa Green Comparison}
The second Green function comparison checks the coordinate-sum operator
\[
A_{\x}^\alpha=-(-\Delta_{x_1})^{\alpha/2}-(-\Delta_{x_2})^{\alpha/2}
\]
on the square $D=(-1,1)^2$, where $\x=(x_1,x_2)$.  The equation is
\[
A_{\x}^\alpha u=\lambda u,
\qquad\mbox{equivalently}\qquad
\big((-\Delta_{x_1})^{\alpha/2}+(-\Delta_{x_2})^{\alpha/2}+\lambda\big)u=0
\quad\mbox{in }D.
\]
Here the manufactured solution is separable:
\[
u_{\mathrm{exact}}(\x)
=G_{\mu,\alpha}(x_1-a_1)\,G_{\nu,\alpha}(x_2-a_2),
\qquad
\mu+\nu=\lambda.
\]
In the reported run
\[
\lambda=0.1,\qquad \mu=\nu=0.05,\qquad
\mathbf{a}=(a_1,a_2)=(2.5,2.25).
\]
This construction is matched to the nonisotropic coordinate-sum operator.  It is not the radial full-space Green kernel of the isotropic fractional Laplacian; rather, it uses the one-dimensional resolvent identity in each coordinate. Since both poles lie outside the corresponding coordinate intervals,
\[
\begin{aligned}
(-\Delta_{x_1})^{\alpha/2}G_{\mu,\alpha}(x_1-a_1)=-\mu G_{\mu,\alpha}(x_1-a_1),
\\
(-\Delta_{x_2})^{\alpha/2}G_{\nu,\alpha}(x_2-a_2)=-\nu G_{\nu,\alpha}(x_2-a_2)
\end{aligned}
\]
for $\x\in D$. Therefore
\[
\left((-\Delta_{x_1})^{\alpha/2}+(-\Delta_{x_2})^{\alpha/2}\right)u_{\mathrm{exact}}
=-(\mu+\nu)u_{\mathrm{exact}}
=-\lambda u_{\mathrm{exact}},
\]
or equivalently $A_{\x}^\alpha u_{\mathrm{exact}}=\lambda u_{\mathrm{exact}}$ in $D$.  The exterior data are again prescribed on the whole complement:
\[
g(\x)=u_{\mathrm{exact}}(\x),
\qquad \x\in D^c.
\]

\begin{table}[H]
\centering
\begin{tabular}{ll}
\hline
Parameter & Value \\
\hline
Domain & $(-1,1)^2$ \\
Fractional order & $\alpha=1.5$ \\
Yukawa parameter & $\lambda=0.1$ \\
Separable parameters & $\mu=\nu=0.05$ \\
Auxiliary frequency & $\omega=\lambda^{1/\alpha}\approx0.215443$ \\
Pole location & $\mathbf{a}=(2.5,2.25)$ \\
Iteration shots & $1{,}000{,}000$ \\
Jump pool size & $1{,}000{,}000$ \\
Grid size & $64\times64$ \\
Random seed & $20260505$ \\
Time step & $10^{-4}$ \\
Exit tolerance & $10^{-5}$ \\
Maximum steps & $20000$ \\
Quadrature cutoff and panels & $300.0,\ 6000$ \\
Green function table size and range & $2049,\ |r|\le60$ \\
\hline
\end{tabular}
\caption{Configuration for the two-dimensional separable Yukawa Green comparison case.}
\end{table}

\begin{figure}[H]
\centering
\includegraphics[width=0.98\textwidth]{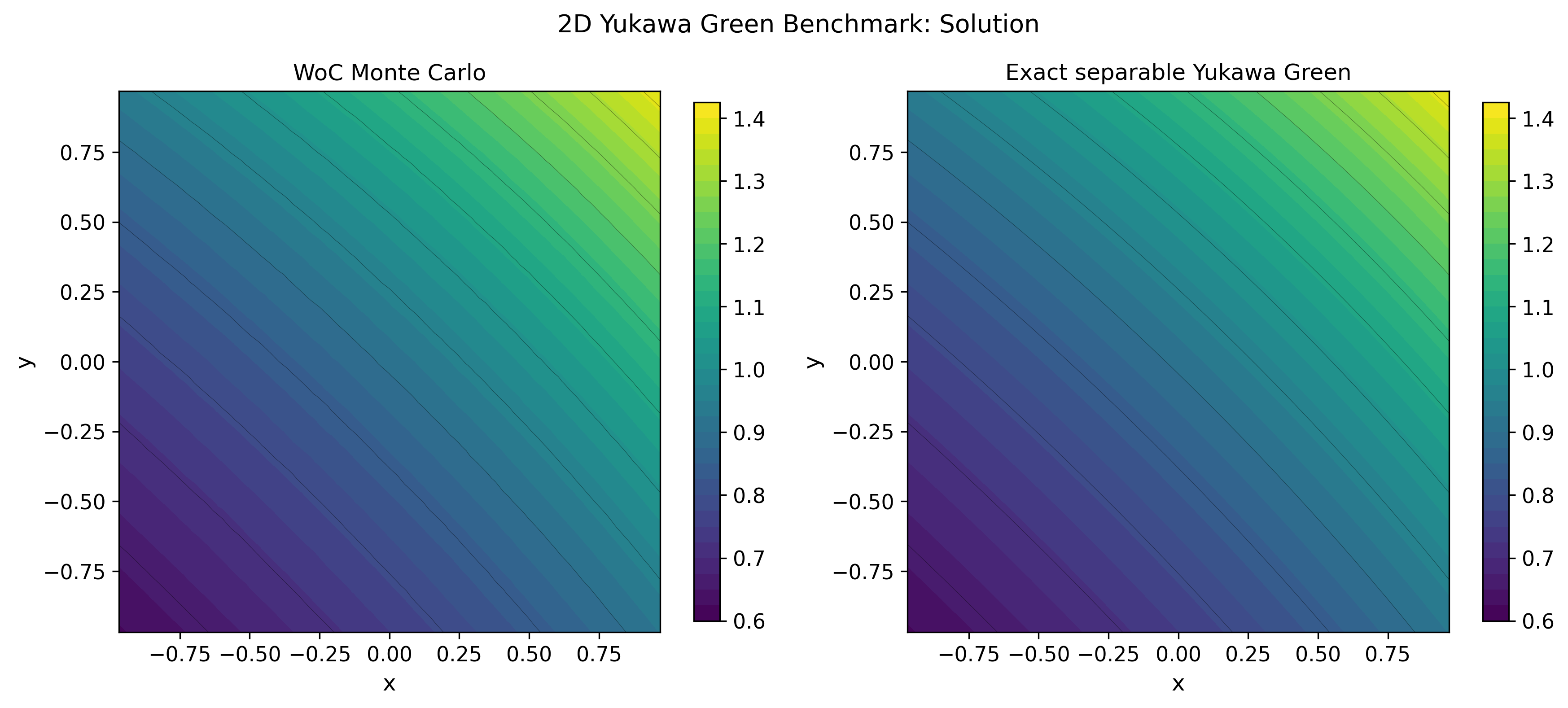}
\caption{Two-dimensional separable Yukawa Green comparison. The WoC Monte Carlo solution is shown beside the exact separable Green function solution.}
\label{fig:yukawa-green-2d-solution}
\end{figure}

\begin{figure}[H]
\centering
\includegraphics[width=0.98\textwidth]{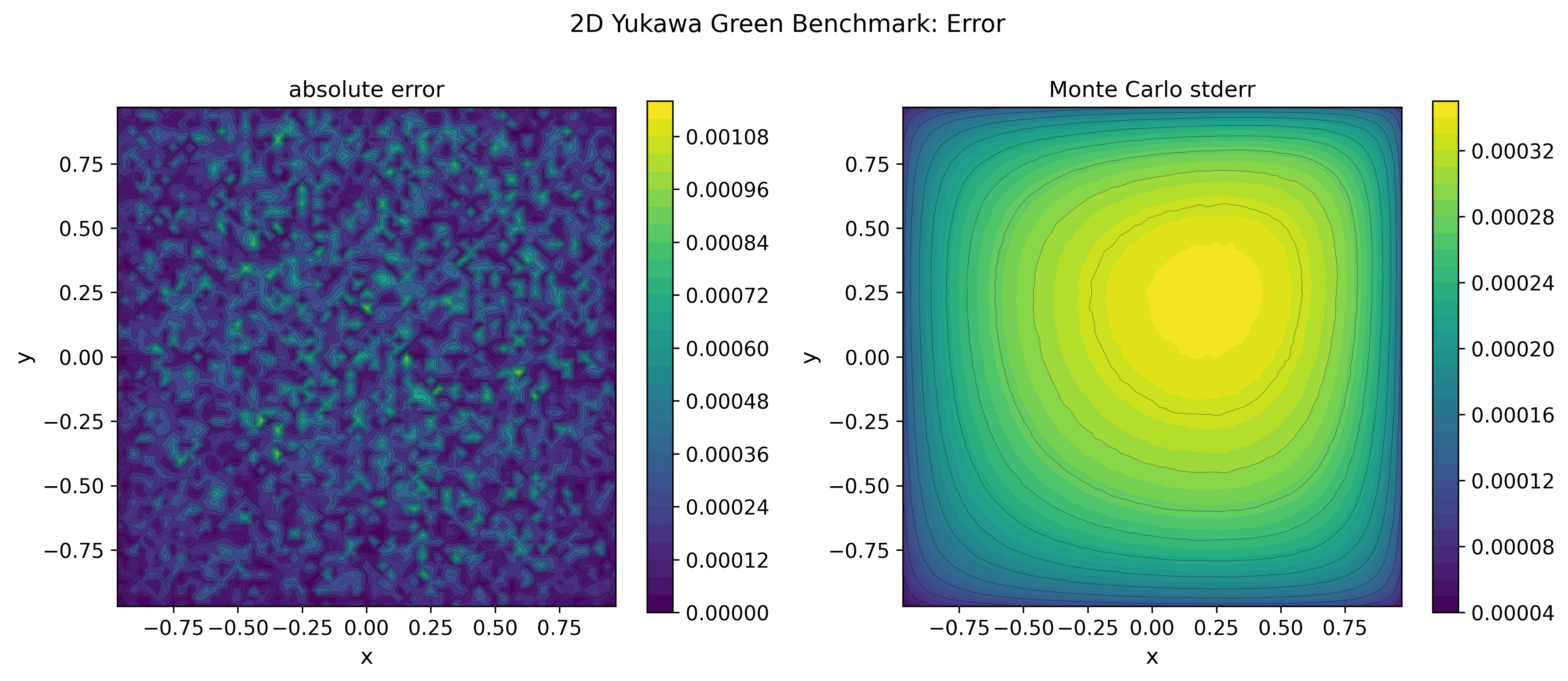}
\caption{Pointwise error for the two-dimensional separable Yukawa Green comparison. The absolute error map is shown beside the Monte Carlo standard error map.}
\label{fig:yukawa-green-2d-error}
\end{figure}

Figures~\ref{fig:yukawa-green-2d-solution}--\ref{fig:yukawa-green-2d-error} show the expected agreement between the WoC estimator and the separable manufactured solution.  The absolute error plot stays at the scale predicted by the Monte Carlo standard error, which supports the use of the Green function construction as a benchmark for the Yukawa regime of the nonisotropic operator.
\subsection{One-Dimensional Helmholtz Walk-on-Intervals Validation}
\label{sec:benchmark-1d-helmholtz}

We first validate the one-dimensional solver against the exact Fourier mode derived in \eqref{eq:report3-real-eigenspace}. For the generator
\begin{equation*}
A^\alpha=-(-\Delta)^{\alpha/2},
\end{equation*}
the manufactured exact solution
\begin{equation*}
u_{\mathrm{exact}}(\x)=\cos(kx_1),\qquad \x=(x_1),
\end{equation*}
satisfies
\begin{equation*}
A^\alpha u_{\mathrm{exact}}(\x)=-k^\alpha u_{\mathrm{exact}}(\x),
\end{equation*}
so the corresponding Helmholtz parameter is $\lambda=-k^\alpha<0$. The computational domain is $D=(-1,1)$, and the exterior Dirichlet data are taken from the exact solution:
\begin{equation*}
g(\x)=u_{\mathrm{exact}}(\x)=\cos(kx_1),\qquad \x\in D^c.
\end{equation*}

Before the production runs, the solver performs the same two stability checks used in the main Helmholtz code path. For the exterior data, the spatial $L^2$ criterion checks the polynomial growth of the exterior boundary data against the fractional exit tail: if $|g(x)|\le C(1+|x|)^p$ at infinity, then the one-dimensional condition $p<\alpha/2$ is sufficient for the sampled exterior contribution to have a finite second moment. The detected growth rate was
\begin{equation*}
p\approx0.401810<\alpha/2=0.750000,
\end{equation*}
so the spatial $L^2$ criterion passed. For the time gauge, the estimated principal eigenvalue and amplification rate were
\begin{equation*}
\lambda_1\approx 1.611357,\qquad -\lambda=0.353553,
\end{equation*}
which places the test in the subcritical regime $-\lambda<\lambda_1$.

\begin{table}[H]
\centering
\begin{tabular}{ll}
\hline
Parameter & Value \\
\hline
Dimension & $1$ \\
Domain & $[-1,1]$ \\
$L$ & $1.0$ \\
Boundary condition & $u(\x)=\cos(kx_1)$ \\
Exact solution & $u(\x)=\cos(kx_1)$ \\
Fractional order & $\alpha=1.5$ \\
Iteration shots & $1{,}000{,}000$ \\
Jump pool size & $50{,}000$ \\
Time step & $3\times 10^{-4}$ \\
$\epsilon$ & $10^{-5}$ \\
Maximum steps & $20{,}000$ \\
Grid size & $64$ \\
\hline
\end{tabular}
\caption{Common benchmark configuration for the two representative \texttt{helmholtz\_1m} runs.}
\label{tab:helmholtz-benchmark-config-main}
\end{table}

The lowest-ratio run corresponds to \texttt{000\_helmholtz\_1m\_solution\_d1}; its shared configuration is listed in Table~\ref{tab:helmholtz-benchmark-config-main}. The parameters are
\begin{equation*}
\frac{|\lambda|}{\lambda_1}=0.15,\qquad
\lambda=-0.2417036195774674,\qquad
k=0.38802118360634363,
\end{equation*}
with random seed \texttt{20269505}. The observed errors were
\begin{align*}
\|u_{\mathrm{MC}}-u_{\mathrm{exact}}\|_{\infty}=0.001297267461,
\\
\|u_{\mathrm{MC}}-u_{\mathrm{exact}}\|_{2}=0.0005009832629540645,
\end{align*}
with mean standard error $0.0003333879742063492$ and zero maximum step hits.

\begin{figure}[H]
\centering
\includegraphics[width=0.82\textwidth]{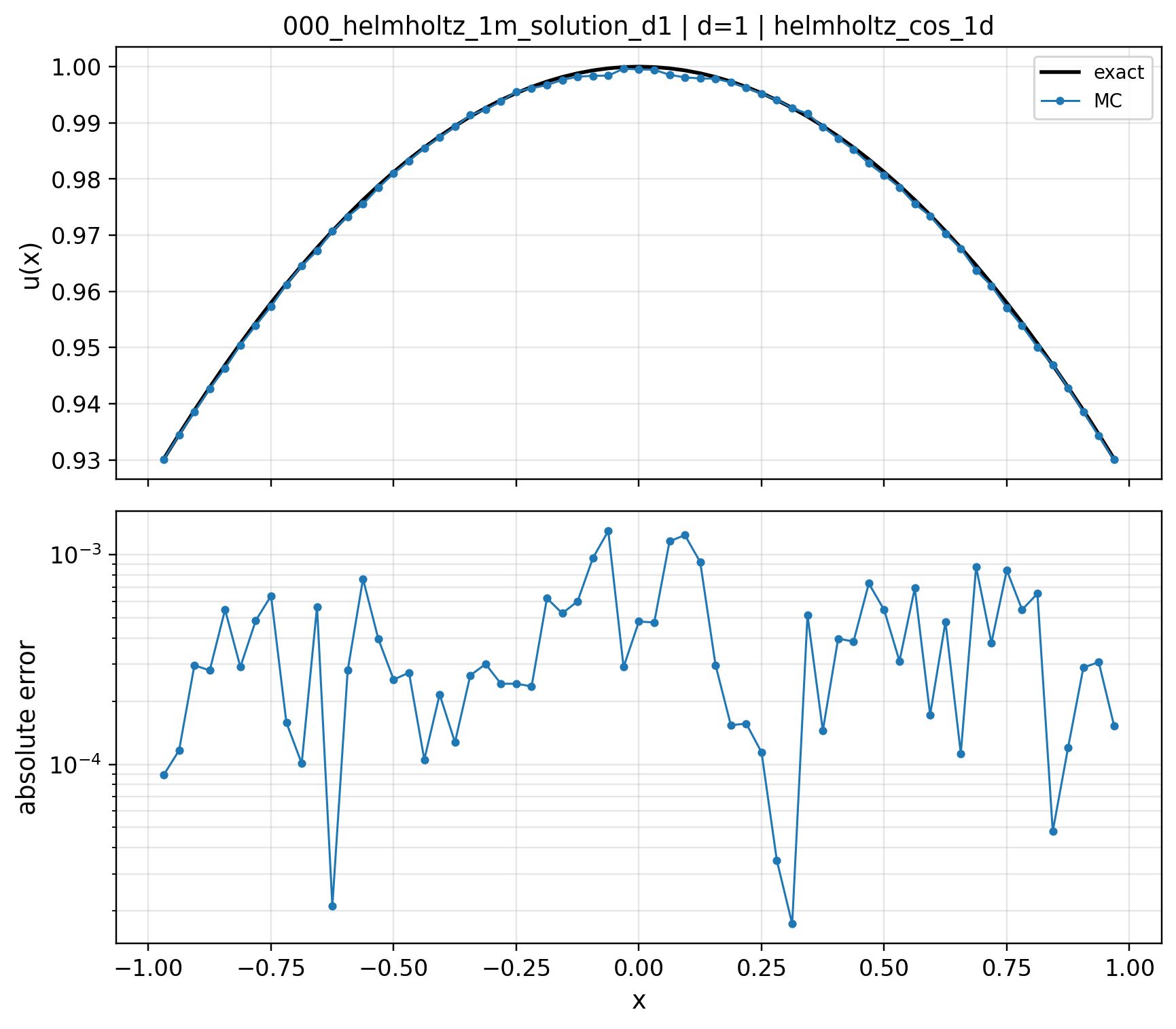}
\caption{Lowest-ratio benchmark run in the \texttt{helmholtz\_1m} profile, corresponding to \texttt{000\_helmholtz\_1m\_solution\_d1} with $|\lambda|/\lambda_1=0.15$.}
\label{fig:benchmark-helmholtz-ratio-015}
\end{figure}

The highest-ratio run corresponds to \texttt{003\_helmholtz\_1m\_solution\_d1}, with
\begin{equation*}
\frac{|\lambda|}{\lambda_1}=0.65,\qquad
\lambda=-1.0473823515023588,\qquad
k=1.0313438926461413,
\end{equation*}
and random seed \texttt{20269508}. Its errors were
\begin{align*}
\|u_{\mathrm{MC}}-u_{\mathrm{exact}}\|_{\infty}=0.052417647075,
\\
\|u_{\mathrm{MC}}-u_{\mathrm{exact}}\|_{2}=0.014328343996343449,
\end{align*}
with mean standard error $0.008218872176079367$ and zero maximum step hits.

\begin{figure}[H]
\centering
\includegraphics[width=0.82\textwidth]{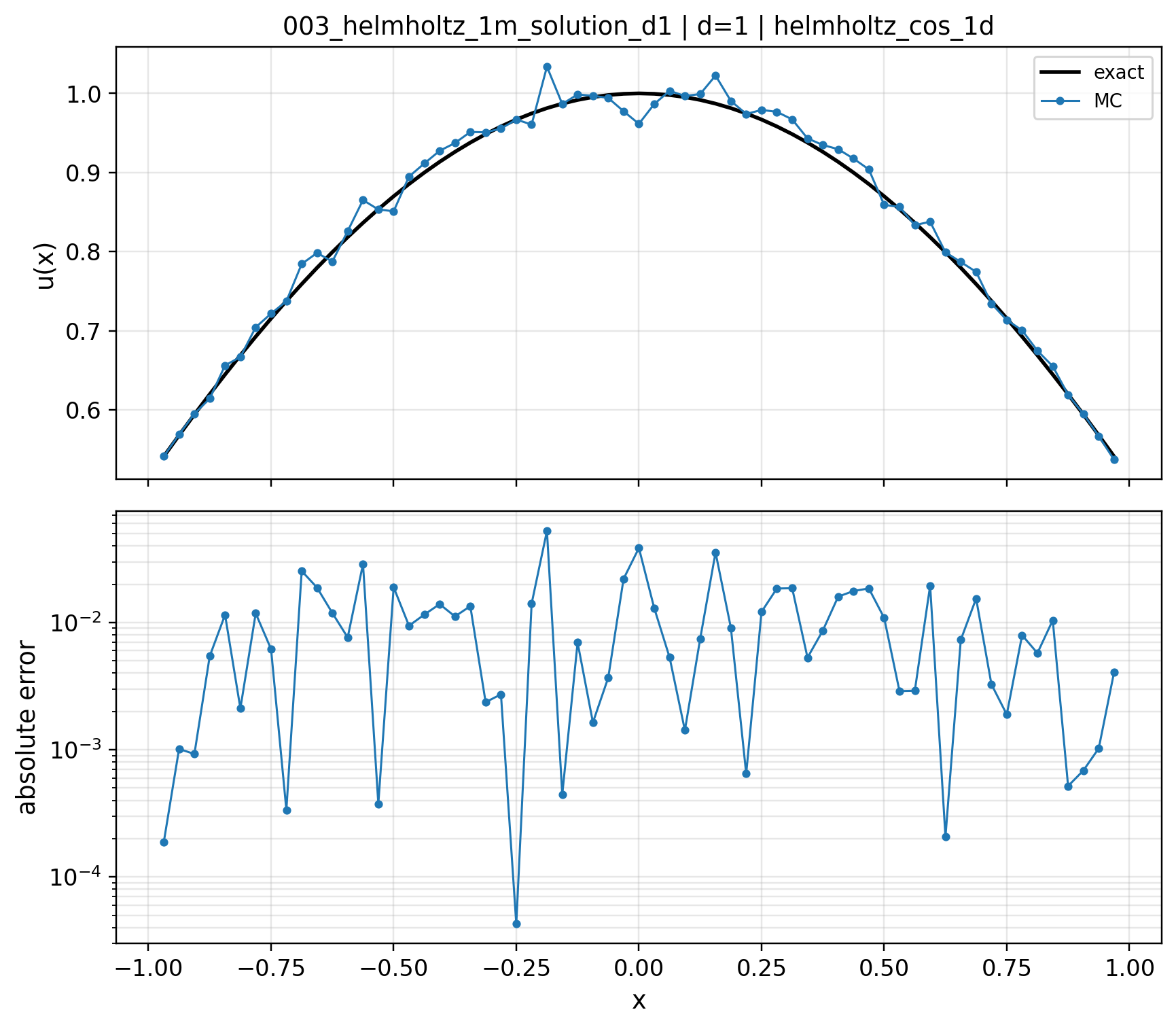}
\caption{Highest-ratio benchmark run in the \texttt{helmholtz\_1m} profile, corresponding to \texttt{003\_helmholtz\_1m\_solution\_d1} with $|\lambda|/\lambda_1=0.65$.}
\label{fig:benchmark-helmholtz-ratio-065}
\end{figure}

Comparing the two runs shows the expected deterioration as $|\lambda|/\lambda_1$ increases toward the critical threshold. The low-ratio case remains highly accurate, while the higher-ratio case exhibits visibly larger bias and variance, consistent with the amplifying Feynman--Kac weight in the Helmholtz regime.
\subsection*{Systematic Dataset Diagnostics}
To complement the individual benchmark plots, we also ran systematic experiments, denoted Experiments 1--3. These experiments cover Monte Carlo convergence, jump pool convergence, and principal eigenvalue survival tail validation.

\begin{table}[H]
\centering
\small
\begin{tabular}{@{}>{\raggedright\arraybackslash}p{0.10\textwidth}>{\raggedright\arraybackslash}p{0.34\textwidth}>{\raggedright\arraybackslash}p{0.48\textwidth}@{}}
\hline
Test & Quantity varied & Main numerical conclusion \\
\hline
Exp1 & $N_{\mathrm{shots}}=10^3,\ldots,10^6$ for the 1D Yukawa Green benchmark & Mean $L^2$ error decreases from $7.234846\times10^{-3}$ to $2.973546\times10^{-4}$; the fitted log--log slope is $-0.464887$, close to the expected Monte Carlo rate. \\
Exp2 & Jump pool size $N_{\mathrm{pool}}=10^3,\ldots,10^6$ for 1D Yukawa, 2D Yukawa, and 1D Helmholtz tests & The finest pool errors are $2.556641\times10^{-4}$, $2.708984\times10^{-4}$, and $1.237600\times10^{-3}$, respectively, showing stabilization once the pool is sufficiently large. \\
Exp3 & Survival tail regression for intervals, rectangles, and hypercubes & The median relative error in $\lambda_1$ is $1.434713\times10^{-2}$, the maximum relative error is $5.962481\times10^{-2}$, and the median tail regression $R^2$ is $0.999956$. \\
\hline
\end{tabular}
\caption{Summary of the trusted systematic dataset diagnostics from Experiments 1--3.}
\label{tab:trusted-systematic-diagnostics}
\end{table}

The Monte Carlo convergence test in Figure~\ref{fig:systematic-exp1-exp2} supports the expected square-root sampling trend.  The fitted slope is slightly shallower than $-1/2$, which is consistent with finite-grid effects, tabulated Green function error, and jump pool discretization being present together with sampling noise.  The jump pool experiment in the same figure shows that very small pools introduce a visible additional error, especially in the Green function tests, while larger pools bring the error down to the sampling level.

\begin{figure}[H]
\centering
\begin{minipage}{0.49\textwidth}
\centering
\includegraphics[width=\textwidth]{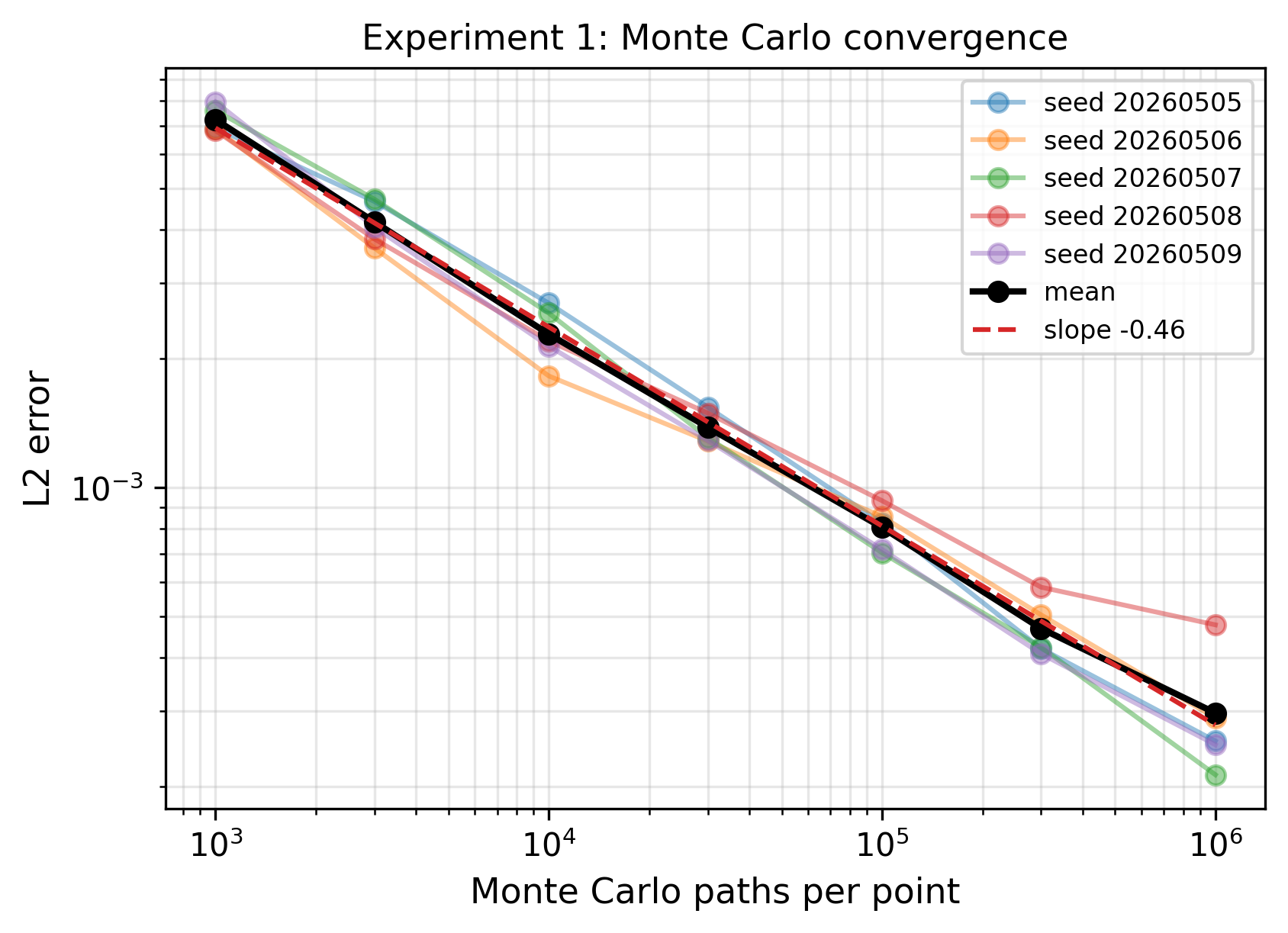}
\end{minipage}
\hfill
\begin{minipage}{0.49\textwidth}
\centering
\includegraphics[width=\textwidth]{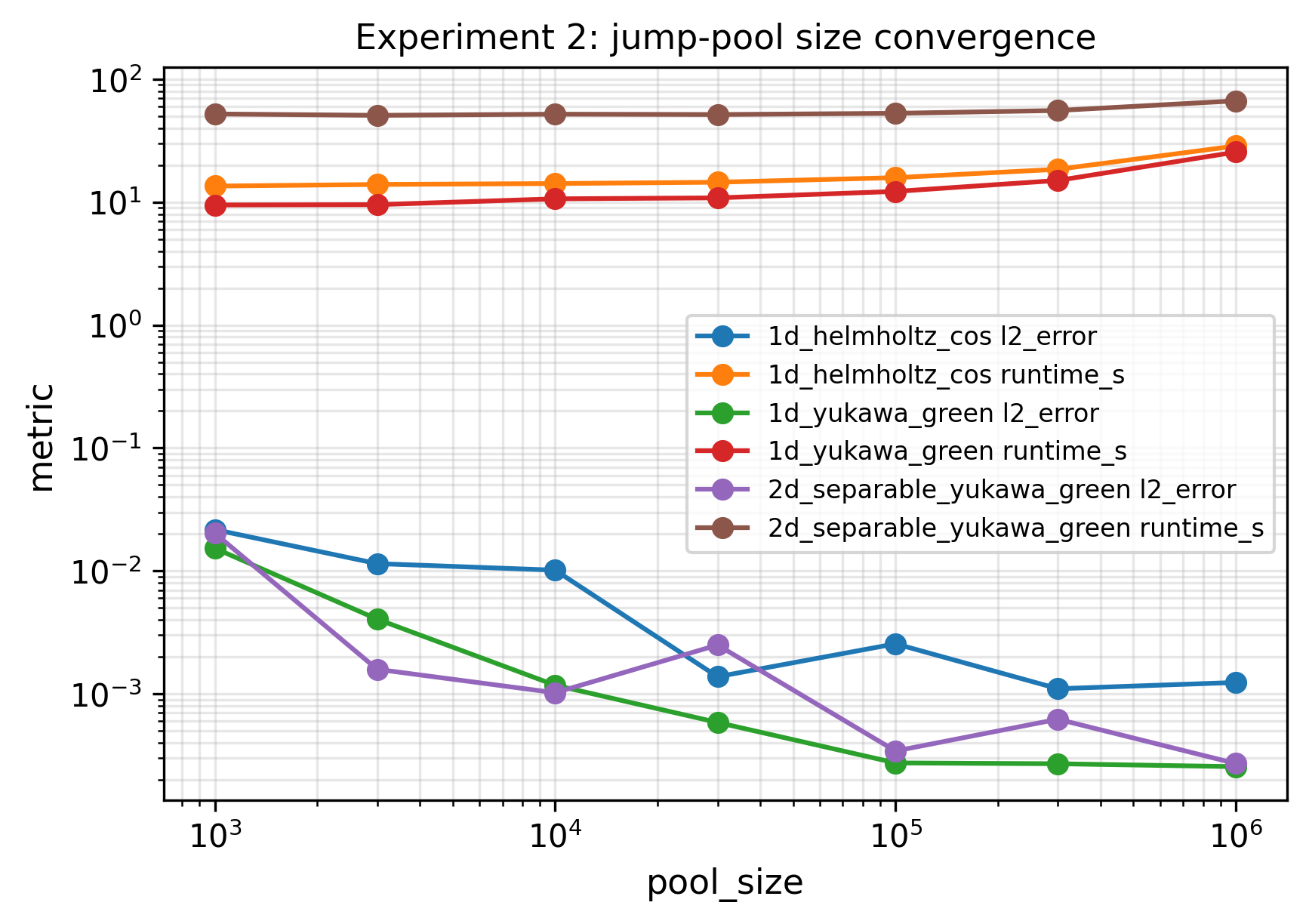}
\end{minipage}
\caption{Trusted dataset diagnostics for Monte Carlo convergence and jump pool convergence.}
\label{fig:systematic-exp1-exp2}
\end{figure}

The eigenvalue survival tail validation in Figure~\ref{fig:systematic-exp3} checks the numerical mechanism used to estimate the principal threshold in the Helmholtz regime.  The tests include interval scaling in one dimension, additivity on product rectangles, and dimension scaling on hypercubes.  The high tail regression $R^2$ values and small relative errors indicate that the survival profiler gives a reliable estimate of the exponential decay rate controlling the Feynman--Kac admissibility threshold.

\begin{figure}[H]
\centering
\includegraphics[width=1.0\textwidth]{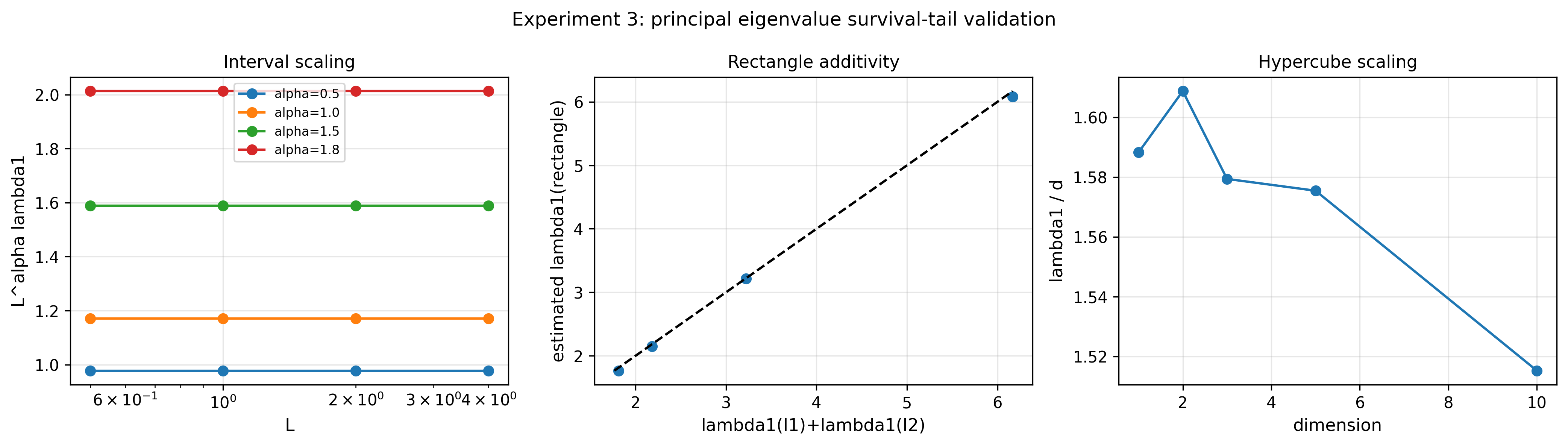}
\caption{Trusted dataset diagnostics for principal eigenvalue estimation from survival tail regression.}
\label{fig:systematic-exp3}
\end{figure}

For the following four Laplace test cases, the common numerical parameters are
\[
\begin{aligned}
\alpha&=1.5, & \lambda&=0, &
N_{\mathrm{shots}}&=10^8\ \text{per grid point},\\
\Delta t_{\mathrm{micro}}&=10^{-4}, &
\epsilon&=10^{-5}.
\end{aligned}
\]

\subsection*{Case 1: 2D Doubly Connected Domain}
This case investigates a non-trivial topological domain containing an internal obstacle.
\begin{itemize}
    \item \textbf{Domain:} $D = [-1, 1]^2 \setminus [-0.5, 0.5]^2$ (A square with a square hole).
    \item \textbf{Boundary Condition:} $g(\x) = \sin(\pi x_1) \cos(\pi x_2)$ for $\x=(x_1,x_2) \in D^c$.
\end{itemize}

\begin{figure}[H]
  \centering
  \includegraphics[width=0.8\textwidth]{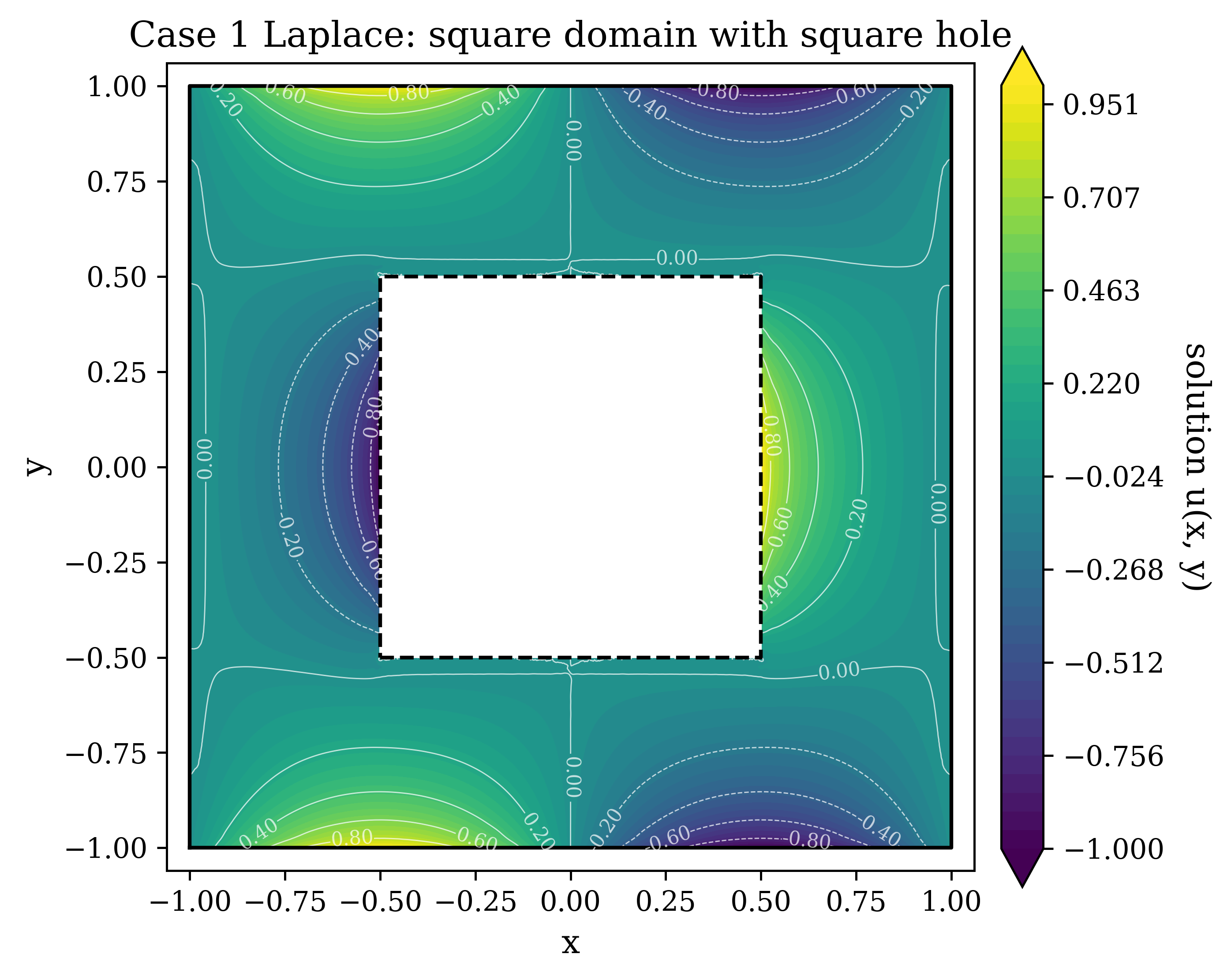}
  \caption{Numerical solution contour for the 2D Doubly Connected Domain. The bounded and oscillatory nature of $g(\x)$ effectively suppresses the variance blow-up induced by heavy-tailed fractional jumps.}
  \label{fig:case1}
\end{figure}

\subsection*{Case 2: 3D Asymmetric Non-convex Domain}
This case tests the algorithm's capability to handle sharp re-entrant corners in 3D without suffering from meshing singularities typical in Finite Element Methods (FEM).
\begin{itemize}
    \item \textbf{Domain:} $D = [-1, 1]^3 \setminus [0, 1]^3$ (A cube missing one octant).
    \item \textbf{Boundary Condition:} $g(\x) = x_1 + x_2 - x_3$ for $\x=(x_1,x_2,x_3) \in D^c$.
\end{itemize}

\begin{figure}[H]
  \centering
  \includegraphics[width=1\textwidth]{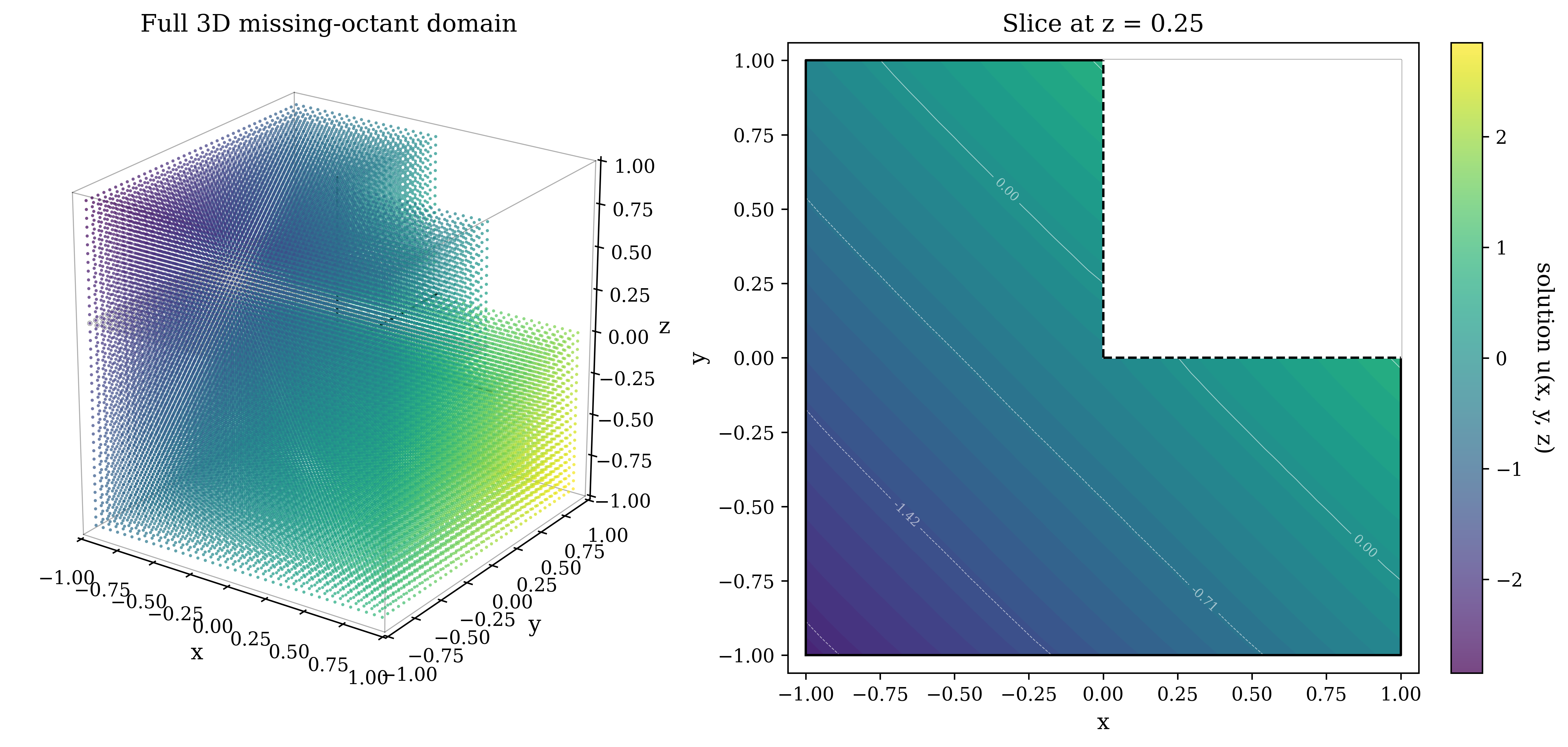}
  \caption{Numerical solution volume rendering for the 3D Missing Octant Domain.}
  \label{fig:case2}
\end{figure}

\subsection*{Case 3: 2D Unit Disk}
Simulating the WoC algorithm on curvilinear domains requires careful formulation of the maximal inscribed $L_\infty$ cube to maintain $O(1)$ jump efficiency.
\begin{itemize}
    \item \textbf{Domain:} $D = \{ \x=(x_1,x_2) \mid x_1^2 + x_2^2 < 1 \}$ (Unit Disk).
    \item \textbf{Boundary Condition:} $g(\x) = \cos(x_1) + \sin(x_2)$ for $\x\in D^c$.
\end{itemize}

\begin{figure}[H]
  \centering
  \includegraphics[width=0.8\textwidth]{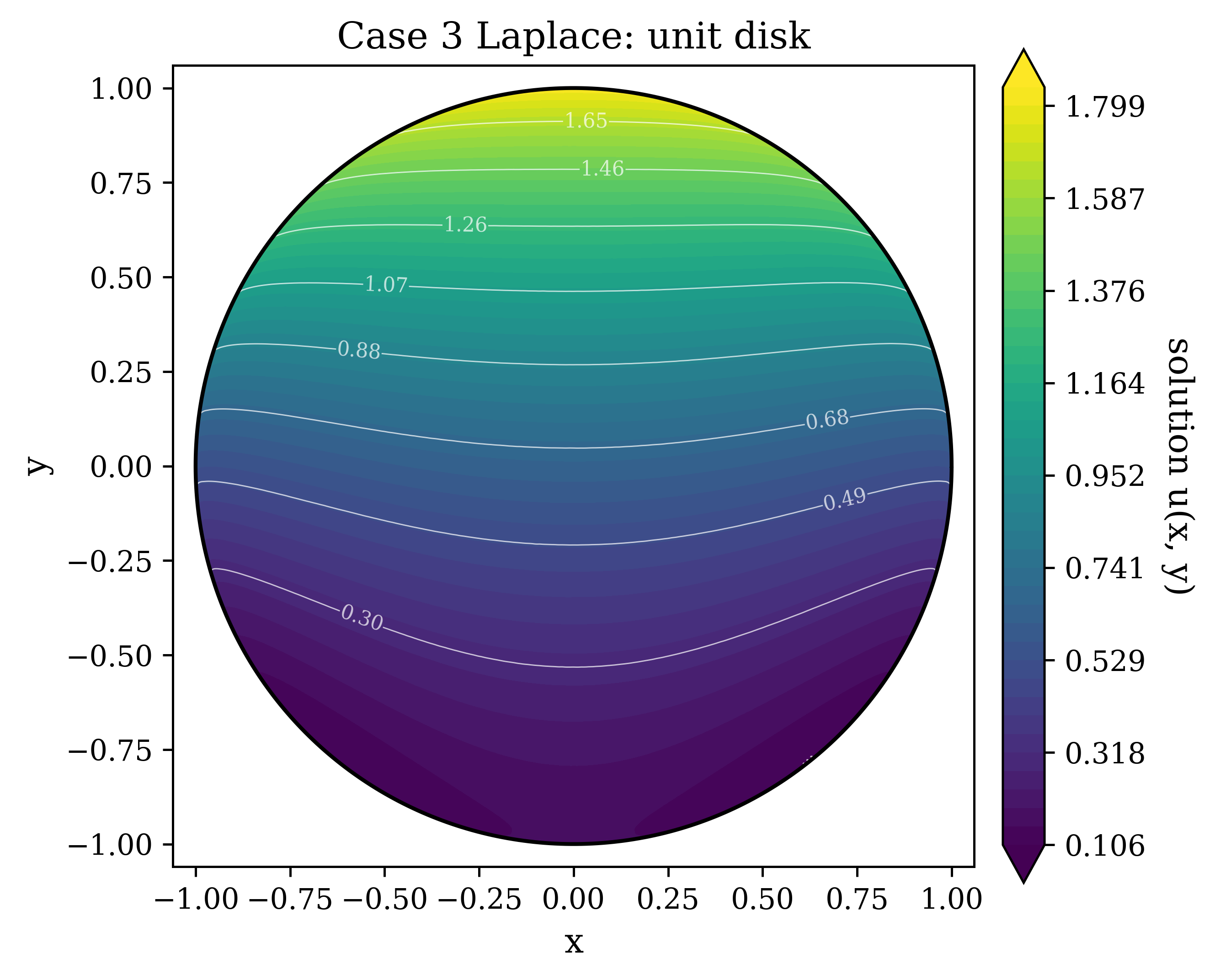}
  \caption{Numerical solution for the 2D Unit Disk Domain.}
  \label{fig:case3}
\end{figure}

\subsection*{Case 4: 3D Spherical Shell}
This case extends the geometric distance derivation to 3D and tests a bounded exponential boundary condition.
\begin{itemize}
    \item \textbf{Domain:} $D = \{ \x=(x_1,x_2,x_3) \mid 0.5 < |\x| < 1.0 \}$ (Spherical Shell).
    \item \textbf{Boundary Condition:} $g(\x) = \exp(-|\x|^2)$ for $\x\in D^c$.
\end{itemize}

\begin{figure}[H]
  \centering
  \includegraphics[width=1\textwidth]{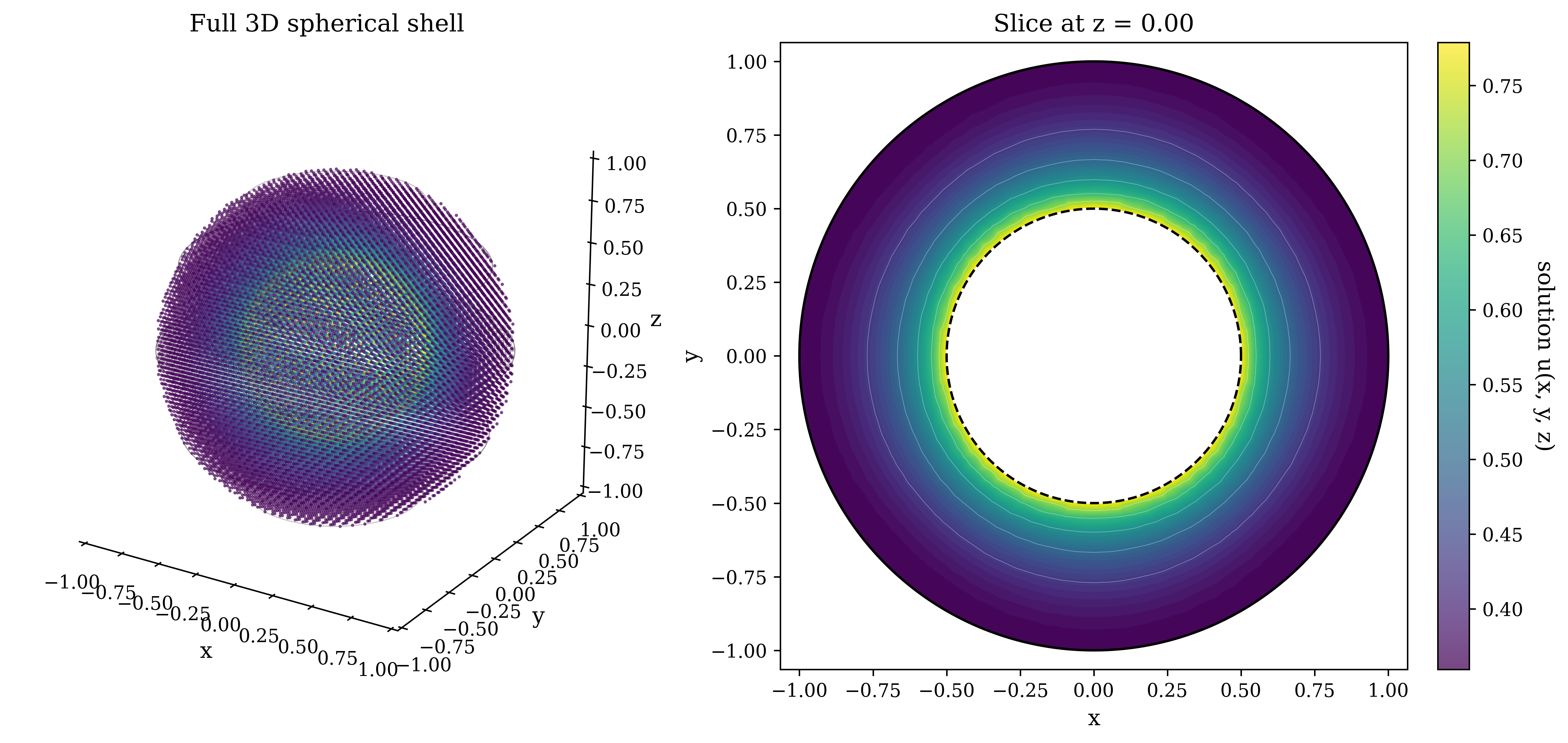}
  \caption{Slice rendering of the numerical solution within the 3D Spherical Shell Domain.}
  \label{fig:case4}
\end{figure}

%%%%%%%%%%%%%%%%%%%%%%%%%%%%%%%%%%%%%%%%%%%%%%%%%%%%%%%%%%%%%%%%%%%%%%%%%%%%%%%
%%%%%%%%%%%%%%%%%%%%%%%%%%%%%%%%%%%%%%%%%%%%%%%%%%%%%%%%%%%%%%%%%%%%%%%%%%%%%%%

%\section{Examples}

%BLAHBLAHBLAH

%%%%%%%%%%%%%%%%%%%%%%%%%%%%%%%%%%%%%%%%%%%%%%%%%%%%%%%%%%%%%%%%%%%%%%%%%%%%%%%
%%%%%%%%%%%%%%%%%%%%%%%%%%%%%%%%%%%%%%%%%%%%%%%%%%%%%%%%%%%%%%%%%%%%%%%%%%%%%%%
\bibliographystyle{siam}
\bibliography{pipliateekki}
\end{document}